\documentclass[11pt,reqno]{amsart}
\usepackage{amsmath,amssymb,amsthm}
\usepackage[T1]{fontenc}
\usepackage{lmodern}
\usepackage[margin=1.15in]{geometry}
\usepackage{microtype}
\usepackage[unicode,colorlinks=true,linkcolor=blue,citecolor=blue,urlcolor=blue]{hyperref}

\hypersetup{pdftitle={Analytic Liouville Integrable Geodesic Flow Has Zero Entropy}}
\numberwithin{equation}{section}
\newtheorem{theorem}{Theorem}[section]
\newtheorem{lemma}[theorem]{Lemma}
\newtheorem{proposition}[theorem]{Proposition}

\theoremstyle{definition}

\theoremstyle{remark}
\newtheorem{remark}[theorem]{Remark}

\newcommand{\R}{\mathbb{R}}

\title[Analytic Liouville Integrable Geodesic Flow Has Zero Entropy]{Analytic Liouville Integrable Geodesic Flow Has Zero Entropy}

\author{Fei Liu}
\address{College of Mathematics and Systems Science, Shandong University of Science and Technology, Qingdao, 266590, China.}
\email{liufei$@$math.pku.edu.cn }
\author{Xiaokai Liu}
\address{Department of Mathematics, Southern University of Science and Technology, Shenzhen, 518055, China.}
\email{liuxk2024@mail.sustech.edu.cn}
\author{Xiang Zhang}
\address{School of Mathematical Sciences, Shanghai Jiao Tong University, Shanghai 200240, China.}
\email{xzhang@sjtu.edu.cn }

\date{\today}

\begin{document}
\begin{abstract}
    This paper studies the topological entropy of real-analytic Liouville integrable Hamiltonian systems. We prove that the Hamiltonian flow on a compact connected regular energy hypersurface has zero topological entropy whenever a complete family of commuting real-analytic first integrals is generically independent in a connected neighborhood of that hypersurface. The proof uses finite normal expansions along minimal analytic carriers and an ergodic version of Ziglin's leading-term argument to obtain independent invariant normal polynomials, whose Poisson geometry makes all tangent Lyapunov exponents of the carriers vanish.

    In particular, we prove that a real-analytic Liouville integrable geodesic flow on a closed real-analytic Riemannian manifold has zero topological entropy, giving a positive answer to the conjecture proposed by Paternain in 1992.
\end{abstract}
\maketitle

\section{Introduction}\label{sec:introduction}

The relation between complete integrability and the complexity of a Hamiltonian flow is a classical problem in dynamics. On each compact connected component of a regular common level of a complete family of commuting first integrals, the Liouville--Arnold theorem describes the flow as a translation on a torus~\cite{Arnold1989}. This description implies zero topological entropy on the regular invariant tori. However, it does not determine the dynamics on the singular levels of the moment map, because an invariant probability measure may be supported entirely on such a level. Thus, to obtain the zero topological entropy from the complete integrability requires information that is not contained in the regular action--angle coordinates.

The topology of the configuration manifold provides another approach to integrability. Kozlov \cite{Kozlov1979} in 1979 proved that an analytic natural mechanical system on a closed orientable surface of genus greater than one has no additional global analytic first integral independent of the energy. In particular, such a surface does not admit an analytically integrable geodesic flow. Taimanov \cite{Taimanov1987} in 1987 confirmed obstructions in arbitrary dimension. He introduced geometric simplicity and proved that it follows from real-analytic integrability. Consequently, the fundamental group of a closed manifold that admits an analytically integrable geodesic flow must be virtually Abelian, and its first Betti number does not exceed its dimension. These results put restrictions on the underlying manifold but do not automatically determine the entropy of the flow on its singular invariant sets.

Paternain \cite{Paternain1991} in 1991 introduced an entropy approach to this problem. He proved that a completely integrable Hamiltonian system with two degrees of freedom has zero topological entropy on a compact regular energy hypersurface if the additional integral is analytic there, or if the connected components of its critical set are submanifolds. His argument in \cite{Paternain1991} combines the existence of horseshoes for non-singular flows of positive entropy in dimension three, with the obstruction to transverse homoclinic trajectories induced by an analytic first integral~\cite{Katok1980, Moser1973}.

In \cite{Paternain1992}, Paternain related the vanishing of the topological entropy of geodesic flows to the topology of the configuration manifolds. In addition to sub-exponential growth of the fundamental group, zero entropy implies rational ellipticity when the fundamental group is finite.
His Conjecture~1.1 of \cite{Paternain1992} in 1992 asks whether, for geodesic flows, complete integrability implies polynomial growth of the fundamental group, and implies rational ellipticity in the case of finite fundamental group. On page~111 of the same paper, he explicitly conjectured that \emph{a real-analytically integrable geodesic flow on a compact manifold has zero topological entropy, while observing that the corresponding smooth assertion seemed unlikely}.
He proved in 1994 zero entropy under the hypothesis of periodic first integrals, and established it for completely integrable Hamiltonians with nondegenerate first integrals \cite{Paternain1994}. The latter result applies to compact invariant sets in arbitrary dimension, but its nondegeneracy assumption is an additional condition on the singular orbits. The same paper explicitly identified the unresolved issue: \emph{in higher dimensions, can the singular set of an analytically integrable system carry positive topological entropy}?

This entropy question is not purely dynamical. Dinaburg's classic result \cite{Dinaburg1971} in 1971 turns exponential growth of the fundamental group into positive topological entropy for geodesic flows. Later, Liu, Wang, and Wu \cite{LWW2020} in 2020 generalized this result to autonomous positive definite Lagrangian systems. When the fundamental group is finite, Paternain ~\cite[Theorem~3.2]{Paternain1992} in 1992 showed that zero topological entropy of the geodesic flow implies rational ellipticity of the manifold. Thus, an affirmative answer supplies both a dynamical obstruction to analytic integrability and strong restrictions on the topology of the underlying manifold.

The distinction between smooth and analytic integrability was first revealed in Butler's examples \cite{Butler1999} in 1999. He constructed smoothly integrable geodesic flows on compact nil-manifolds whose fundamental groups are not virtually Abelian. These examples have zero topological entropy, while the topological restrictions for analytic integrability need not hold in the smooth category. Bolsinov and Taimanov \cite{BTInventiones2000} in 2000 then constructed a real-analytic metric on a compact three-dimensional Sol-manifold whose geodesic flow is Liouville integrable with respect to smooth first integrals and has positive topological entropy. In their example, every invariant smooth probability measure with a positive density with respect to volume has zero metric entropy. A singular invariant subsystem carries the positive entropy, which is a suspension of a hyperbolic toral automorphism containing a transverse two-torus. The time-one map preserves this torus, and its restriction to the torus is the hyperbolic automorphism.

One of the additional smooth integral in Bolsinov and Taimanov's construction contain a flat factor multiplying an oscillatory function of a logarithmic momentum coordinate. Such functions extend smoothly across the singular set but are not analytic there. The impossibility of an alternative complete analytic family of first integrals follows from Taimanov's topological obstruction. Thus, the analyticity of the metric alone cannot replace the analyticity of the integrals. In their subsequent paper \cite{BT2000} in 2000, Bolsinov and Taimanov extended the construction and restated the real-analytic version of Paternain's conjecture.

Further developments illustrate the range of smooth integrable chaos. Butler \cite{Butler2004} in 2004 related integrable systems with positive entropy on torus bundles to periodic Toda lattices, obtaining geodesic Hamiltonians for particular families. Liu, Chen, and Zhang \cite{LCZ2009} in 2009, Liu and Zhang \cite{LZ2010} in 2010, and Matveev \cite{Matveev2012} in 2012 considered natural Hamiltonians on torus suspensions with a periodic potential in the base direction. Their constructions give smoothly integrable systems of positive entropy above the potential barrier. More recently, Bolsinov, Veselov, and Ye \cite{Bolsinov2021} in 2021 studied geodesic flows on quotients of $\operatorname{PSL}(2,\R)$. They found two separate open regions. The flow is analytically integrable in one region, while in the other region, it has positive entropy but admits no additional independent smooth first integral. Thus, integrability on one open region does not establish the existence of a complete analytic family that is generically independent on a connected neighborhood of the entire energy hypersurface.

The geometry of singular integrable systems has also been studied extensively. Zung ~\cite{Zung1996} in 1996 established torus actions near nondegenerate singular leaves and, under an additional topological stability assumption, topological decompositions up to finite coverings. He also obtained torus actions for certain degenerate corank-one singularities~\cite{Zung2000} in 2000 and near compact singular orbits satisfying a finite-type condition~\cite{Zung2003} in 2003. His analytic normalization theorem~\cite{Zung2005} in 2005 requires no nondegeneracy assumption and establishes convergent Birkhoff normalization near an equilibrium. The geometric linearization theorems~\cite[Theorems~4.3 and~4.5]{Zung2015} by Zung in 2015 concern nondegenerate singularities of integrable dynamical systems. Here we study entropy on a compact regular energy hypersurface while allowing arbitrary degeneracies of the analytic moment map.

This question will be settled down in this paper through finite normal expansions of the first integrals. We prove the following Hamiltonian result and its consequence for geodesic flows.

\begin{theorem}\label{thm:hamiltonian-entropy}
    Let $(P^{2n},\omega)$ be a real-analytic symplectic manifold and $H\in C^\omega(P)$. For a regular value $e$ of $H$, let $\Sigma$ be a compact connected component of $H^{-1}(e)$.  Suppose that there are real-analytic functions
    \begin{displaymath}
        F_1=H,\ F_2,\ldots,F_n
    \end{displaymath}
    defined on a connected open neighborhood $U\subset P$ of $\Sigma$ such that
    \begin{displaymath}
        \{F_i,F_j\}=0\quad(1\leq i,j\leq n),\qquad d_PF_1\wedge\cdots\wedge d_PF_n\not\equiv0\quad\text{on} ~U.
    \end{displaymath}
    If $\varphi^t$ is the Hamiltonian flow of $H$, then
    \begin{displaymath}
        h_{\mathrm{top}}(\varphi^1|_\Sigma)=0.
    \end{displaymath}
\end{theorem}

The independence assumption in Theorem~\ref{thm:hamiltonian-entropy} is made in the ambient neighborhood $U$. In particular, we do not assume that the given additional integrals are generically independent after restriction to $\Sigma$, or even that the ambient exterior product is nonzero at some point of $\Sigma$.

Geodesic flows form an important class of Hamiltonian systems. Let $(M^n,g)$ be a complete Riemannian manifold. Its geodesic flow $\phi^t$ is defined for every $t\in\mathbb{R}$ by
\begin{displaymath}
    \phi^t:TM\longrightarrow TM,\qquad (x,v)\longmapsto \left(c_v(t),c_v'(t)\right).
\end{displaymath}
where $c_v$ is the unique geodesic satisfying $c_v(0)=x$ and $c_v'(0)=v$. Since geodesics have constant speed, $\phi^t$ preserves the unit tangent bundle $T^1M$. Geodesics satisfy the Euler--Lagrange equations associated with the kinetic energy Lagrangian
\begin{displaymath}
    L(x,v)=\frac12 g_x(v,v)=\frac12{\|v\|}^2.
\end{displaymath}
This Lagrangian is strictly convex and superlinear in the velocity variable, and $T^1M=L^{-1}(\frac{1}{2})$. The Legendre transform is the diffeomorphism
\begin{displaymath}
    \mathcal{L}:TM\longrightarrow T^*M,\qquad (x,v)\longmapsto \bigl(x,g_x(v,\cdot)\bigr).
\end{displaymath}
The corresponding Hamiltonian is
\begin{displaymath}
    H(x,p)=\frac{1}{2} g_x^{-1}(p,p).
\end{displaymath}
Here $g_x^{-1}$ denotes the dual metric on $T_x^*M$. In local coordinates, its matrix is the inverse of the symmetric positive definite matrix representing $g_x$. The Legendre transform conjugates the geodesic flow to the Hamiltonian flow of $H$ on $T^*M$, equipped with its canonical symplectic form. It maps $T^1M$ onto the unit cotangent bundle $S^*M=H^{-1}(\frac12)$. See~\cite{Arnold1989, Paternain1999} for more details.

\begin{theorem}\label{thm:geodesic-entropy}
    Let $(M^n,g)$ be a closed, connected real-analytic Riemannian manifold, where the dimension $n\geq2$, and let
    \begin{displaymath}
        H(x,p)=\frac12g_x^{-1}(p,p),\qquad (x,p)\in T^*M,
    \end{displaymath}
    be its geodesic Hamiltonian. Suppose that there exist a connected open neighborhood $U\subset T^*M$ of the unit cotangent bundle $S^*M=H^{-1}(1/2)$ and real-analytic functions
    \begin{displaymath}
        F_1=H,\ F_2,\ldots,F_n
    \end{displaymath}
    which Poisson commute and satisfy
    \begin{displaymath}
        dF_1\wedge\cdots\wedge dF_n\not\equiv0\quad\text{on }U.
    \end{displaymath}

    Then the unit speed geodesic flow $\phi^t$ on $T^1M$ has zero topological entropy:
    \begin{displaymath}
        h_{\mathrm{top}}(\phi^1)=0.
    \end{displaymath}
\end{theorem}

Theorem~\ref{thm:geodesic-entropy} gives an affirmative answer to the real-analytic version of Paternain's conjecture of \cite{Paternain1992} and restated in \cite{BT2000}. It allows arbitrary singularities of the moment map and does not require nondegeneracy or generic independence on the chosen energy hypersurface.

We clarify the independence convention. On a connected real-analytic domain, the condition $dF_1\wedge\cdots\wedge dF_n\not\equiv0$ is equivalent to independence on an open dense set and to independence almost everywhere with respect to any smooth positive density. This is the same meaning of functional independence as used almost everywhere by Bolsinov and Taimanov \cite[pp.~640--641]{BTInventiones2000}, when the comparison is made in the same domain.

There is a distinction between this convention in an ambient neighborhood and its energy-level version. If $\Sigma=H^{-1}(e)$ is regular, then at each $x\in\Sigma$,
\begin{displaymath}
    dH\wedge dF_2\wedge\cdots\wedge dF_n\big|_x\ne0 \quad\Longleftrightarrow\quad d(F_2|_\Sigma)\wedge\cdots\wedge d(F_n|_\Sigma)\big|_x\ne0.
\end{displaymath}

Here, the exterior product on the left is an ambient covector, not its pullback to $\Sigma$. Ambient generic independence need not imply generic independence of the chosen restrictions at every energy. For example, replacing a commuting integral $G$ by $(H-e)G$ may preserve ambient generic independence while making the restriction identically zero on $\Sigma$.  Theorem~\ref{thm:hamiltonian-entropy} considers this possibility.

The proof starts with an ergodic invariant measure for the time-one map. Choose one connected embedded real-analytic submanifold of the energy component of the least dimension, which is invariant under the flow with full measure. This carrier need not be closed or compact, and the measure itself need not be invariant under the flow. Minimality implies that every proper, relatively closed analytic subset of the carrier that is invariant under the flow has measure zero. This principle shows that the first integrals vanish identically on the carrier after subtracting their almost everywhere constant values. Analyticity then provides finite, nonzero terms in their full ambient normal expansions. The resulting polynomial family is equivariant with respect to the normal derivative cocycle. This extends the classical leading-term construction for first integrals of variational equations \cite[\S1.1]{Ziglin1982} to normal expansions along the invariant carrier.

The algebraic step is related to Ziglin's argument on the leading term \cite[\S1.2]{Ziglin1982}. For a fixed germ, Ziglin replaces functionally independent integrals by polynomial combinations so that their leading homogeneous terms become independent. In the present setting, the expansion is along a carrier of possibly positive dimension, and the leading polynomials vary with the base point. Hence, pointwise polynomial dependence does not immediately yield a single polynomial relation among the original integrals. We use ergodicity to ensure that the finite dimensional spaces of relations are constant almost everywhere. Their associated analytic sections have invariant zero sets that are relatively closed,  so the minimal-carrier zero-set principle extends these relations to the entire carrier. The unique continuation theorem would not justify this step for a possibly singular measure.  Adapting the descent argument in Ziglin's lemma
\cite[\S1.2]{Ziglin1982} (see also \cite[\S III.6, pp.~75--76]{Audin2008}), we prove that the polynomial replacement procedure terminates by constructing a strictly decreasing nonnegative integer from the normal orders of the integrals and their full ambient differential exterior product. This gives the dynamical version of the leading term argument required here.

Finally, the Poisson relations and the Lyapunov properties of the independent normal polynomials make the tangent Lyapunov exponents on the carrier vanish. The entropy argument is applied to the measure carried by this submanifold, rather than to all ambient exponents, since transverse exponents can be nonzero even in an integrable system of zero entropy. The variational principle then gives Theorem~\ref{thm:hamiltonian-entropy}. The geodesic statement follows by applying it to the unit cotangent bundle and using the Legendre transform. The detailed argument is given in Section~\ref{sec:proof-main-results}.

The role of analyticity is clear when compared with the smooth examples. Set
\begin{displaymath}
    \rho(s)=\begin{cases}
        \exp(-1/s^2),&s\ne0,\\
        0,&s=0.
    \end{cases}
\end{displaymath}

In suitable momentum coordinates, the Bolsinov--Taimanov integral \cite{BTInventiones2000} has the form
\begin{displaymath}
    \rho(p_up_v)\sin\!\left(\frac{2\pi\log|p_u|}{\log\lambda}\right)
\end{displaymath}
away from $p_up_v=0$, where $\lambda>1$ is an eigenvalue of the monodromy, and extends smoothly by zero across that set. It is flat there: all of its normal derivatives vanish. Thus, the integral contributes to the functional independence of this family away from the singular set. Now consider a carrier of a measure of positive entropy that is contained in that set. Along any analytic carrier contained in this set, every normal jet of this integral vanishes. Thus, it supplies no nonzero leading normal polynomial. The analytic argument also allows the entire measure to be carried by the singular locus, but infinite order vanishing cannot hide an analytic integral in this way. Finite normal orders and the full ambient exterior product permit polynomial replacements until the leading terms are independent. Their Poisson commutation constrains the resulting polynomial differentials, while invariance forces the leading tensors to have zero total Lyapunov weight. Together, these facts imply zero tangent exponents on the carrier; zero total weights alone would not suffice. The flat smooth integral supplies no leading tensor to which this argument could be applied. This explains both why singular measures must be treated and why smooth integrability cannot replace analytic integrability in the theorems.

\section{The Proof of the Main Results}\label{sec:proof-main-results}

\subsection{Analytic carriers and ambient normal expansions}\label{sec:analytic-carriers}

Throughout the proof of Theorem~\ref{thm:hamiltonian-entropy}, we use the conventions
\begin{displaymath}
    \iota_{X_J}\omega=d_PJ, \qquad \{J,K\}=\omega(X_J,X_K).
\end{displaymath}
Thus $X_HJ=\{J,H\}$.  Unless a different space is indicated, differentials and orders of vanishing are computed in $U$, not in the energy hypersurface. We write $f=\varphi^1|_\Sigma$.

We will prove that every $f$-ergodic probability measure has zero entropy.

The main point is to study the first nonzero normal terms of the integrals along an analytic submanifold carrying such a measure. The carrier may lie entirely in the singular set of the original moment map. For this reason, independence on a full ambient volume set does not by itself settle the problem.

\subsubsection{The ambient flow along the energy hypersurface}

Since $X_HH=0$, the Hamiltonian vector field is tangent to $\Sigma$. Compactness implies that its restriction to $\Sigma$ is complete. Hence $f$ is a real-analytic diffeomorphism of the compact manifold $\Sigma$. Moreover, $d_PH\ne0$ on $\Sigma$, so the Hamiltonian vector field $X_H$ has no zeros on $\Sigma$.

We also need the derivative of the flow in the full tangent bundle of $P$. Fix $T>0$.  For each $x\in\Sigma$, the orbit segment $\{\varphi^t(x):|t|\leq T\}$ is contained in the compact set $\Sigma \subset U$. Local existence, uniqueness, and continuous dependence for the differential equation give an ambient neighborhood of $x$ on which the solutions exist for $|t|\leq T$ and stay in $U$. A finite cover of $\Sigma$ gives an open neighborhood $V_T\subset U$ with this property. The neighborhood $V_T$ may depend on $T$.

Consequently, for every $x\in\Sigma$ and every $t\in\mathbb R$,
\begin{equation}\label{eq:ambient-derivative}
    A^t(x):=D\varphi^t_x:T_xP\longrightarrow T_{\varphi^t x}P
\end{equation}
is an invertible symplectic map (i.e., preserves the symplectic form $\omega$) and satisfies the cocycle identity. If $\{J,H\}=0$ on $U$, then $J\circ\varphi^t=J$ on the corresponding local flow domain. In particular, this identity holds for ambient germs along $\Sigma$.

\subsubsection{A minimal analytic carrier}

Fix an arbitrary $f$-ergodic Borel probability measure $\mu$ on $\Sigma$. Note that only $f$-invariance and $f$-ergodicity are assumed for $\mu$, and they do not imply invariance under the Hamiltonian flow $\varphi^t$.

Consider the connected embedded real-analytic submanifolds $N'\subset \Sigma$ satisfying:
\begin{displaymath}
    \varphi^t(N')=N'\quad(t\in\mathbb R),\qquad \mu(N')=1.
\end{displaymath}
This collection is nonempty because $\Sigma$ belongs to it. Its dimensions belong to a finite set of integers. Choose a member of the least dimension, denote it by $N$, and put $d=\dim N$. Since an embedded submanifold is a Borel subset, the condition $\mu(N')=1$ is meaningful. Note that the manifold $N$ need not be closed or compact. Since it is invariant under the flow, $X_H(x)\in T_xN$, and therefore $d\geq1$.

\begin{lemma}\label{lem:carrier-zero-set}
    Let $Z$ be a proper, relatively closed real-analytic subset of $N$. If $\varphi^t(Z)=Z$ for all $t\in\mathbb R$, then $\mu(Z)=0$. Consequently, if an analytic section of an analytic vector bundle over $N$ has a flow-invariant zero set and vanishes $\mu$-almost everywhere, then it vanishes identically on $N$.
\end{lemma}

\begin{proof}
    First, we show that a proper relatively closed analytic subset of the connected analytic manifold $N$ has empty interior.

    To see this directly, let $E=\operatorname{int}_N Z$. By definition, $E$ is an open subset of $N$. Suppose that $E \neq \emptyset$. Then, for each point $p \in \overline{E}^{N}$ (the closure of $E$ in $N$), since $Z$ is an analytic subset, we can choose a connected local analytic coordinate chart (neighborhood) $V$ such that
    \begin{displaymath}
        Z \cap V = \{ g_{1}=\cdot\cdot\cdot=g_{r} = 0\},
    \end{displaymath}
    where $g_{i} \in C^{\omega}(V)$ for each $i \in \{1,\dots,r\}$. Since $p \in \overline{E}^{N}$ and $V$ is a neighborhood of $p$, $V \cap E \neq \emptyset$. While both $E$ and $V$ are open subsets in $N$, $V \cap E$ is a non-empty open subset of $N$, and
    \begin{displaymath}
        E \cap V \subseteq  Z \cap V = \{ g_{1}=\cdot\cdot\cdot=g_{r} = 0\}.
    \end{displaymath}
    Thus, each $g_{i}$ vanishes identically on $E \cap V$. By the connectedness of $V$ and the real analyticity of $g_{i}$, we know that $g_{i}\equiv 0$ on $V$. Therefore
    \begin{displaymath}
        V \subseteq \{ g_{1}=\cdot\cdot\cdot=g_{r} = 0\}=Z \cap V \subseteq Z.
    \end{displaymath}
    Since $V$ is a neighborhood of the point $p$, the above inclusion implies that $p \in E$. Thus $E = \overline{E}^{N}$, i.e., $E$ is both open and closed in $N$, and the connectedness of $N$ implies that either $E = \emptyset$ or $E = N$. If $E = N$, then $Z = N$, which contradicts the fact that $Z$ is a proper subset of $N$. Therefore $E =\operatorname{int}_N Z= \emptyset$, i.e., $Z$ has no interior points. Furthermore, by Bierstone-Milman's stratification theory \cite{BierstoneMilman1988}, we have $\dim Z<d =\dim N$.

    Second, we prove that $Z$ can be decomposed into countably many connected, flow-invariant analytic pieces.

    We will use the analytic regular-locus theorem and the local stratification theorem for semianalytic sets \cite[Corollary~2.11, Theorem~7.2, and Remark~7.3]{BierstoneMilman1988} in the following intrinsic form. For a semianalytic set $X$ of dimension $r$, let
    \begin{displaymath}
        \mathcal R_r(X)=\{x\in X\mid X~\text{is an embedded analytic}~r\text{-manifold near}~x\}.
    \end{displaymath}
    Here ``analytic manifold'' is essential; see \cite[Definitions~3.3 and~7.1]{BierstoneMilman1988}. The set $\mathcal R_r(X)$ is relatively open and semianalytic, and
    \begin{displaymath}
        \dim\left(X\setminus\mathcal R_r(X)\right)<r.
    \end{displaymath}
    We recall the dimension argument to make clear how the stratification is used. Take a compatible locally finite analytic stratification with the frontier condition. At a point of an $r$-dimensional stratum $S$, no other stratum of $X$ can have that point in its frontier: the frontier condition would require a stratum of dimension strictly larger than $r$. Local finiteness permits a neighborhood meeting only finitely many strata, and shrinking this neighborhood away from the closures of the other strata makes $X$ locally equal to $S$. Hence, every top-dimensional stratum lies in $\mathcal{R} _ {r} (X)$, and the complement has a smaller dimension. Semianalyticity of the regular locus is the additional conclusion of \cite[Remark~7.3]{BierstoneMilman1988}.

    Starting with $X_0=Z$, define, as long as $X_j\ne\varnothing$,
    \begin{displaymath}
        r_j=\dim X_j,\qquad X_{j+1}=X_j\setminus\mathcal R_{r_j}(X_j).
    \end{displaymath}

    By Bierstone-Milman's theory as stated above, $\mathcal R_{r_j}(X_j)$ is relatively open and semianalytic, and
    \begin{displaymath}
        \dim\left(X_{j}\setminus\mathcal R_{r_j}(X_j)\right)<r_{j}.
    \end{displaymath}
    Since the dimensions decrease strictly, this procedure stops after finitely many steps. Each connected component of $\mathcal R_{r_j}(X_j)$ is an embedded analytic submanifold. The connected components of the sets $\mathcal R_{r_j}(X_j)$ form a countable decomposition of $Z$ into connected embedded analytic submanifolds of dimension less than $d$. Countability follows because each regular locus is a second-countable manifold and its connected components are open. See \cite{BierstoneMilman1988} for more details.

    This decomposition is defined by the local analytic structure of the set, rather than by a choice of charts. Every analytic diffeomorphism preserving $Z$ therefore preserves each $X_j$ and each regular locus $\mathcal R_{r_j}(X_j)$. In particular, the flow does so. If $x$ belongs to a connected component $C$ of one of these regular loci, the path $t\mapsto\varphi^t(x)$ remains in that regular locus and starts in $C$. It stays in $C$. Applying the same argument to negative time shows $\varphi^t(C)=C$. Thus $Z$ can be decomposed into countably many connected, flow-invariant analytic pieces $C_{l}$; each $C_{l}$ satisfies $\dim C_{l}\leq \dim Z < d$.

    Third, we show that $\mu(Z)=0$.

    Since $Z$ is an $f$-invariant Borel set, the ergodicity shows that $\mu(Z)\in\{0,1\}$. If $\mu(Z)=1$, by the fact that $Z = \bigcup_{l}C_{l}$, i.e., $Z$ is a countable union, there exists an $l_{0}$ such that $\mu(C_{l_{0}})>0$. Since $C_{l_{0}}$ is flow-invariant, it is $f$-invariant, while $\mu$ is an ergodic measure of $f$; this implies that $\mu(C_{l_{0}})=1$. Thus $C_{l_{0}}$ is a connected, flow-invariant, embedded real-analytic submanifold with $\mu(C_{l_{0}})=1$, and $\dim C_{l_{0}} < d =\dim N$. This contradicts the minimal choice of $d$. Hence $\mu(Z)=0$.

    Finally, in a local analytic frame, the zero set of an analytic section is the common zero set of its finitely many coefficients. It is a relatively closed analytic subset. If the section were not identically zero, the assertion just proved would contradict its almost-everywhere vanishing.
\end{proof}

\subsubsection{The integrals are constant on the carrier}

Each $F_i|_\Sigma$ is an $f$-invariant measurable function. For each $i=2,\ldots,n$, ergodicity therefore gives a constant $c_i\in\mathbb R$ such that $F_i=c_i$ $\mu$-almost everywhere.  The set
\begin{displaymath}
    Z_i=\{x\in N:F_i(x)=c_i\}
\end{displaymath}
is relatively closed, analytic, invariant under the flow, and has full measure. Lemma~\ref{lem:carrier-zero-set} gives $Z_i=N$. In particular, $c_1=e$. Set
\begin{displaymath}
    A_i=F_i-c_i\quad(1\leq i\leq n),\qquad \Omega=d_PA_1\wedge\cdots\wedge d_PA_n.
\end{displaymath}
Then
\begin{equation}\label{eq:current-list-properties}
    A_i|_N=0,\qquad \{A_i,A_j\}=0,\qquad \{A_i,H\}=0,\qquad \Omega\not\equiv0~\text{on}~U.
\end{equation}

\subsubsection{Order of vanishing in the full ambient space}

Let $W=TP|_N\Big/TN$ be the normal bundle of $N$; its rank is $q=\operatorname{rank}W=2n-d$. For $x\in N$, write $\mathcal O_{U,x}$ for the local ring of real-analytic function germs and $\mathcal I_{N,x}$ for the ideal of germs vanishing on the germ of $N$ at $x$.

Since $N$ is an embedded real-analytic submanifold, for each $x_{0}\in N$, there exist an open neighborhood $U_{x_{0}}\subset U$ and adapted analytic local coordinates $(y,z)\in\mathbb R^d\times\mathbb R^q$, such that
\begin{displaymath}
    x_{0}=(0,0),\qquad N \cap U_{x_{0}}=\{z=0\}.
\end{displaymath}
In these coordinates, any $g \in \mathcal O_{U,x_{0}}$ can be written as
\begin{displaymath}
    g(y,z)=\sum_{\alpha\in \mathbb{N}^{q}}c_{\alpha}(y)z^{\alpha},
\end{displaymath}
where $c_{\alpha}$ is a real-analytic function defined in a neighborhood $V$ of $0 \in\mathbb R^d$. The condition $f\mid_{N}=0$ is equivalent to the fact that $c_{0}(y)=0$. Thus there exist $g_{j}\in \mathcal O_{U,x_{0}}$ ($j=1,\ldots,q$), such that
\begin{displaymath}
    g(y,z)=\sum^{q}_{j=1}z_{j}g_{j}(y,z).
\end{displaymath}
Therefore we have
\begin{equation}\label{eq1}
    \mathcal I_{N,x_{0}} = (z_{1},\ldots,z_{q})\cdot \mathcal O_{U,x_{0}}.
\end{equation}

For any nonnegative integer $a \in \mathbb{N}$, in the adapted coordinates above, the $a$-th power of $\mathcal I_{N,x}$ can be written as
\begin{displaymath}
    \begin{aligned}
        \mathcal I_{N,x}^{a} & =  \mathcal I_{N,x} \cdot \mathcal I_{N,x}\cdot\cdot\cdot\mathcal I_{N,x}  \\
                             & = \Big(z^{\beta}~\big |~|\beta|=a\Big)\mathcal O_{U,x}.        \\
    \end{aligned}
\end{displaymath}
In particular, $\mathcal I_{N,x}^{0}=\mathcal O_{U,x}$.

More generally, let $\mathcal{E} \to U$ be an analytic vector bundle, and let $s$ be an analytic section. Its normal order along $N$ is
\begin{equation}\label{eq:order-definition}
    \operatorname{ord}_N s=\sup\{a\in\mathbb N: s_x\in\mathcal I_{N,x}^{a}\mathcal E_x~\text{for every}~x\in N\},
\end{equation}
where $\mathcal E_x$ on the right denotes the module of analytic section germs. We put $\operatorname{ord}_N0=+\infty$.  Multiplication by an invertible analytic change-of-frame matrix preserves the condition in \eqref{eq:order-definition}, so it is independent of the chosen frame.

\begin{lemma}\label{lem:finite-normal-order}
    A nonzero analytic section on the connected set $U$ has finite normal order along $N$.
\end{lemma}

\begin{proof}
    Fix $x_{0}\in N$ and adapted analytic coordinates $(y,z)\in\mathbb R^d\times\mathbb R^q$ centered at $x_{0}$ as above. Let $e_{1},e_{2},\ldots,e_{r}$ be a local analytic frame of $\mathcal E$ on $U_{x_{0}}$; in this frame $s$ can be written as
    \begin{displaymath}
        s=\sum^{r}_{j=1}s_{j}e_{j},\qquad s_{j}\in C^{\omega}(U_{x_{0}}),~j=1,\ldots,r.
    \end{displaymath}
    Each $s_{j}$ can be written in the following form in the adapted coordinates
    \begin{displaymath}
        s_{j}(y,z)=\sum_{\alpha\in \mathbb{N}^{q}}c_{j,\alpha}(y)z^{\alpha},\qquad j=1,\ldots,r,
    \end{displaymath}
    where $c_{j,\alpha}$ is an analytic function on the neighborhood of $0 \in \mathbb{R}^{d}$. Furthermore, $s$ can be written as
    \begin{displaymath}
        s=\sum^{r}_{j=1}\Big(\sum_{\alpha\in \mathbb{N}^{q}}c_{j,\alpha}(y)z^{\alpha} \Big)e_{j},\qquad  s_{j}\in C^{\omega}(U_{x_{0}}),~j=1,\ldots,r.
    \end{displaymath}

    Suppose that $\operatorname{ord}_N s =+\infty$. By definition, for each $a \in \mathbb{N}$, we have $s_{x_{0}}\in\mathcal I_{N,x_{0}}^{a}\mathcal E_{x_{0}}$. As we have seen in \eqref{eq1}, in the adapted coordinates, $\mathcal I_{N,x_{0}} = \big(z_{1},\ldots,z_{q}\big)$, and
    \begin{displaymath}
        \mathcal I_{N,x_{0}}^{a}\mathcal E_{x_{0}} = \left\{ \sum^{r}_{j=1}f_{j}e_{j}~\mid~f_{j} \in (z_{1},\ldots,z_{q})^{a} \right\}.
    \end{displaymath}

    Thus,
    \begin{displaymath}
        \begin{aligned}
            s_{x_{0}}\in\mathcal I_{N,x_{0}}^{a}\mathcal E_{x_{0}} & \Longleftrightarrow  s_{j} \in (z_{1},\ldots,z_{q})^{a}  \\
                                                                   & \Longleftrightarrow c_{j,\alpha}(y) = 0, \quad |\alpha|< a.\\
        \end{aligned}
    \end{displaymath}
    The fact that $\operatorname{ord}_N s =+\infty$ implies that the above equality is valid for every $a\in \mathbb{N}$; thus for all $j$ and $\alpha$, we have $c_{j,\alpha}(y) = 0$. Furthermore $s \equiv 0$ in $U_{x_{0}}$, where $U_{x_{0}}\subset U$ is the adapted coordinate neighborhood of $x_{0}$ above. By the identity theorem for real-analytic functions, we get that $s \equiv 0$ in $U$, which contradicts the assumption that $s$ is not identically zero.
\end{proof}

In particular,
\begin{displaymath}
    a_i:=\operatorname{ord}_N A_i\in\mathbb N_{>0},\qquad w:=\operatorname{ord}_N\Omega\in\mathbb N
\end{displaymath}
are finite. Here $\Omega$ is a section of $\Lambda^nT^*U$. It is not the pullback of that section to $N$ or to $\Sigma$. For the initial list, the regularity of the energy gives
\begin{displaymath}
    \operatorname{ord}_N(H-e)=1.
\end{displaymath}

\subsubsection{The leading normal polynomial and its transformation law}\label{subsec:1}

Suppose $J$ is an analytic function of finite normal order $a\geq1$. In adapted coordinates, write:
\begin{displaymath}
    J(y,z)=\sum_{|\alpha|\geq a}c_\alpha(y)z^\alpha,\qquad p_y(z)=\sum_{|\alpha|=a}c_\alpha(y)z^\alpha.
\end{displaymath}
For $x=(y,0)$ and $v\in W_x$, represented in these coordinates by its normal component $v_z$, define
\begin{equation}\label{eq:first-normal-polynomial}
    P_J(x;v)=p_y(v_z)=\frac{1}{a!}\frac{d^a}{dt^a}|_{t=0}J(y,tv_z).
\end{equation}

We identify a homogeneous polynomial with the symmetric multilinear form $T$ normalized by $T(v,\ldots,v)=P_J(x;v)$. Thus, the expression in \eqref{eq:first-normal-polynomial} is the polynomial value, not an unnormalized derivative tensor. For the function $H-e$, we have
\begin{displaymath}
    P_{H-e}(x;v)=d_PH_x(v_{x}).
\end{displaymath}

The last formula is well-defined because $d_PH$ annihilates $T_xN$. It is nonzero because $d_PH_x\ne0$.

To check the independence of coordinates, let
\begin{displaymath}
    (y,z)\in\mathbb R^d\times\mathbb R^q, \qquad (y',z')\in\mathbb R^d\times\mathbb R^q
\end{displaymath}
be two adapted charts, i.e.,
\begin{equation}\label{eq3}
    N = \{z=0\} = \{z'=0\}.
\end{equation}
The coordinate transformations are given by
\begin{equation}\label{eq4}
    y'=\psi(y,z),\qquad
    z'=\xi(y,z).
\end{equation}

By~\eqref{eq3}, we know $\xi(y,0)=0$; thus
\begin{displaymath}
    \xi(y,z)=\xi(y,0) + D_{z}\xi(y,0)z + O(|z|^{2})=0+ B(y)z + O(|z|^{2}),
\end{displaymath}
where $B(y):=D_{z}\xi(y,0)$. Furthermore, the coordinate transition $\Psi=(\psi,\xi)$ is a real-analytic diffeomorphism. The fact that $\xi(y,0)=0$ implies that its derivative along $N$ is
\begin{displaymath}
    D\Psi(y, 0)=
    \begin{pmatrix}
        D\psi_{0}(y) & D_{z}\psi(y,0)\\
        0 & B(y)
    \end{pmatrix},
\end{displaymath}
which is invertible, where $\psi_{0}(y):=\psi(y,0)$. Thus both $D\psi_{0}(y)$ and $B(y)$ are invertible.

Similarly, we have
\begin{displaymath}
    \psi(y,z)=\psi(y,0) + O(|z|)=\psi_{0}(y)+ O(|z|),
\end{displaymath}
where, as stated above, $\psi_{0}(y):=\psi(y,0)$.

Thus we can write~\eqref{eq4} in the following form
\begin{equation}\label{eq5}
    y'=\psi_{0}(y)+ O(|z|),\qquad
    z' = B(y)z + O(|z|^{2}).
\end{equation}
Its inverse has the form
\begin{equation}\label{eq6}
    y=\psi_0^{-1}(y')+O(|z'|),\qquad
    z=B(\psi_0^{-1}(y'))^{-1}z'+O(|z'|^2).
\end{equation}

Suppose $a=\operatorname{ord}_NJ$. In these coordinates, $J$ can be written as
\begin{displaymath}
    J(y,z)=p_{y}(z) + O(|z|^{a+1}),
\end{displaymath}
\begin{displaymath}
    J(y',z'):=J(\Psi^{-1}(y',z'))=p_{y'}(z') + O(|z'|^{a+1}).
\end{displaymath}
Then by the fact that $p_{y}(z)$ and $p_{y'}(z')$ are homogeneous polynomials of degree $a$ and~\eqref{eq5},~\eqref{eq6}, for the points on $N$, we can get that
\begin{equation}\label{eq:first-term-transformation}
    p'_{\psi_0(y)}(z')=p_y\left(B(y)^{-1}z'\right).
\end{equation}
Formula~\eqref{eq:first-term-transformation} proves that
\begin{displaymath}
    P_J\in\Gamma^\omega\left(N,\operatorname{Sym}^aW^*\right)
\end{displaymath}
is an intrinsic analytic section. Equivalently, it is the leading class in
\begin{displaymath}
    \mathcal I_N^a/\mathcal I_N^{a+1}\simeq
    \mathcal O_N(\operatorname{Sym}^a W^*),
\end{displaymath}
where the right-hand side denotes the sheaf of analytic sections of $\operatorname{Sym}^a W^*$. This section is not identically zero. However, it can be the zero polynomial at individual base points. Its zero set, understood as the zero set of the section,
\begin{displaymath}
    Z_{P_J}=\{x\in N:P_J(x;\cdot)\equiv 0\},
\end{displaymath}
is invariant under the Hamiltonian flow given that $\{J,H\}=0$. For $a=0$, the leading term is simply $J|_N$.

The same construction applies to a section $s$ of $\mathcal E$: if $b=\operatorname{ord}_Ns$, its leading term is a section of $\operatorname{Sym}^bW^*\otimes\mathcal E|_N$. In particular, for the ambient $n$-form $\Omega$, the second tensor factor is $(\Lambda^nT^*U)|_N$, not $\Lambda^nT^*N$.

\begin{lemma}\label{lem:order-rules}
    If $\operatorname{ord}_NJ=a\geq1$, then
    \begin{equation}\label{eq:derivative-order}
        \operatorname{ord}_N(d_PJ)\geq a-1.
    \end{equation}

    If $J$ is a nonzero analytic function and $s$ is a nonzero analytic section on $U$, then
    \begin{equation}\label{eq:product-order}
        \operatorname{ord}_N(Js)
        =\operatorname{ord}_NJ+\operatorname{ord}_Ns.
    \end{equation}
\end{lemma}

\begin{proof}
    In the full coframe $dy_1,\ldots,dy_d,dz_1,\ldots,dz_q$, normal differentiation lowers normal degree by at most one, while tangential differentiation does not lower it. This gives \eqref{eq:derivative-order}.

    Let the two orders in \eqref{eq:product-order} be $a$ and $b$. The product has order at least $a+b$, and its degree-$a+b$ term is $P_J(x;v)P_s(x;v)$. Each of the two leading sections has a proper analytic zero set in the connected manifold $N$. Their nonzero sets are open and dense, so they have a common point $x_0$. At $x_0$, choose a nonzero component of the vector-valued polynomial $P_s(x_0;\cdot)$. The product of that component with the nonzero polynomial $P_J(x_0;\cdot)$ is nonzero, because a real polynomial ring has no zero divisors. Thus, the leading product section is not identically zero, which proves equality.
\end{proof}

\subsubsection{Equivariance of the leading normal terms}

Since for all $t\in \mathbb{R}$, we have $\varphi^{t}N=N$, the derivative $D\varphi^{t}$ preserves $TN$, which induces an invertible quotient map
\begin{displaymath}
    C^t(x):W_x\longrightarrow W_{\varphi^t x},\qquad C^t(x)[v]=[D\varphi^t_xv].
\end{displaymath}

As we have seen before, let $a=\operatorname{ord}_NJ$. In the adapted local coordinates $(y,z)\in\mathbb R^d\times\mathbb R^q$, $J$ has the form
\begin{displaymath}
    J(y,z)=p_{y}(z) + O(|z|^{a+1}),
\end{displaymath}
where $p_{y}$ is the leading homogeneous term of degree $a$. The corresponding intrinsic section is
\begin{displaymath}
    P_{J}(x;z)=p_{y}(z), \quad x=(y,0)\in N.
\end{displaymath}
In these coordinates, the Hamiltonian flow is given by
\begin{displaymath}
    \varphi^{t}(y,z)=(\varphi_{y}^{t}(y,z), \varphi_{z}^{t}(y,z)).
\end{displaymath}
Since $\varphi^{t}N=N$, we know that $\varphi_{z}^{t}(y,0)=0$; thus
\begin{displaymath}
    \varphi_{z}^{t}(y,z)=C^{t}(y)z + O(|z|^{2}).
\end{displaymath}
Furthermore
\begin{displaymath}
    \begin{aligned}
        J(\varphi^{t}(y,z)) & =  J(\varphi_{y}^{t}(y,z), \varphi_{z}^{t}(y,z))  \\
                            & =  J(\varphi_{y}^{t}(y,z), C^{t}(y)z + O(|z|^{2}) )           \\
                            & =   p_{\varphi_{y}^{t}(y,z)}(C^{t}(y)z) + O(|z|^{a+1}),          \\
    \end{aligned}
\end{displaymath}

Let $ h_{t}(y):= \varphi_{y}^{t}(y,0)$. Fix $t\in\mathbb R$ and use adapted charts near $x$ and $\varphi^{t}x$; then the flow can be written as
\begin{displaymath}
    \varphi^{t}(y,z)=\left(h_{t}(y)+O(|z|),\,C^{t}(y)z+O(|z|^{2})\right).
\end{displaymath}
Let $p_{y}$ and $\tilde p_{y'}$ be the leading normal polynomials of $J$ in the source and target charts, respectively. Since their coefficients depend analytically on the base point and they are homogeneous of degree $a$ in the normal variable,
\begin{displaymath}
    J\circ\varphi^{t}(y,z)
    =\tilde p_{h_{t}(y)}\left(C^{t}(y)z\right)+O(|z|^{a+1}).
\end{displaymath}
If $\{J,H\}=0$, the local identity $J\circ\varphi^{t}=J$ gives
\begin{displaymath}
    \tilde p_{h_{t}(y)}\left(C^{t}(y)z\right)=p_{y}(z).
\end{displaymath}
In intrinsic notation, this is
\begin{equation}\label{eq:polynomial-equivariance}
    P_{J}\left(\varphi^{t}x;C^{t}(x)v\right)=P_{J}(x;v).
\end{equation}
There is no scalar multiplier in this identity. Its zero set, understood as the zero set of the section,
\begin{displaymath}
    Z_{P_J}=\{x\in N:P_J(x;\cdot)\equiv 0\},
\end{displaymath}
is invariant under the Hamiltonian flow. Since $P_{J}$ is not identically zero, Lemma~\ref{lem:carrier-zero-set} also gives $P_{J}(x;\cdot)\ne0$ for $\mu$-almost every $x$.

\subsection{Weighted relations and finite elimination}\label{sec:finite-elimination}

We now prove that a finite sequence of polynomial replacements makes the leading normal terms algebraically independent in almost every fiber. The coefficients of each replacement must be real constants. An arbitrary relation chosen separately in each fiber would not define an analytic first integral on $U$.

For the current list of integrals, retain \eqref{eq:current-list-properties} and write
\begin{displaymath}
    a_i=\operatorname{ord}_NA_i\geq1,\qquad P_i=P_{A_i}.
\end{displaymath}
The weights $a_i$ are defined along all of $N$ and are not redefined at exceptional points where a particular leading polynomial vanishes.

\subsubsection{Weighted decomposition in a fixed fiber}

Fix a point $x\in N$. We introduce formal variables $u_1,\ldots,u_n$ and give each $u_i$ a weight $a_i$ ($1\leq i \leq n$).  They are placeholders for the polynomials $P_i(x;\cdot)$, not coordinates of $W_x$.  For a multi-index $\beta\in\mathbb N^n$, write
\begin{displaymath}
    u^\beta=\prod_{i=1}^nu_i^{\beta_i},\qquad|\beta|=\sum_i\beta_i,\qquad|\beta|_a=\sum_i a_i\beta_i.
\end{displaymath}
The ordinary degree $|\beta|$ and the weighted degree $|\beta|_a$ need not agree.

Every polynomial $R(u)=\sum_\beta c_\beta u^\beta$ has a unique finite weighted decomposition
\begin{displaymath}
    R=\sum_{D=0}^{D_{\max}}R_D,
    \qquad R_D(u)=\sum_{|\beta|_a=D}c_\beta u^\beta.
\end{displaymath}
Since each $a_i \geq 1$, we know that $|\beta|_a=0~\Leftrightarrow~\beta=0$; therefore $R_0$ is a constant. Each summand satisfies
\begin{displaymath}
    R_D(t^{a_1}u_1,\ldots,t^{a_n}u_n)=t^DR_D(u).
\end{displaymath}
On the other hand, $P_i(x;tv)=t^{a_i}P_i(x;v)$.  Therefore
\begin{displaymath}
    Q_D(v):=R_D(P_1(x;v),\ldots,P_n(x;v))
\end{displaymath}
is an ordinary homogeneous polynomial of degree $D$ on $W_x$. We note that $Q_D(v)$ is possibly the zero polynomial. This follows term by term:
\begin{displaymath}
    \prod_iP_i(x;tv)^{\beta_i}
    =t^{\sum_i a_i\beta_i}\prod_iP_i(x;v)^{\beta_i}
    =t^D\prod_iP_i(x;v)^{\beta_i}
    \quad\text{if }|\beta|_a=D.
\end{displaymath}

Suppose now that $R\ne0$ is a polynomial relation among the $P_i(x)$. Substituting $tv$ into that relation gives
\begin{displaymath}
    0=R(P_1(x;tv),\ldots,P_n(x;tv))
    =\sum_Dt^DQ_D(v)
    \quad(t\in\mathbb R,\ v\in W_x).
\end{displaymath}
For a fixed $v$, the last expression is a one-variable polynomial in $t$ which vanishes identically. Hence $Q_D(v)=0$ for every $D$.  Since $v$ is arbitrary,
\begin{displaymath}
    R_D(P_1(x;\cdot),\ldots,P_n(x;\cdot))\equiv0
    \quad\text{for every }D.
\end{displaymath}
This says that each component becomes zero after substitution, not that each formal polynomial $R_D$ is zero. Setting $v=0$ shows $R_0=0$. As $R\ne0$, at least one nonzero component $R_D$ has $D\geq1$. Thus, every algebraic dependence has a nonzero relation of positive weighted degree.

For example, $P_1(z)=z^2$ and $P_2(z)=z^3$ have weights $(2,3)$. The relation $u_2^2-u_1^3$ has weighted degree $6$, although its two terms have different ordinary degrees in $u$. The polynomial $(1+u_1)(u_2^2-u_1^3)$ has nonzero weighted components of degrees $6$ and $8$, and both vanish separately after substitution. Ordinary homogeneous decomposition in the variables $u$ would not have this property.

\subsubsection{Relation spaces with a fixed coefficient domain}

For $D\geq1$, define the finite-dimensional real vector space
\begin{displaymath}
    V_D=\operatorname{span}_{\mathbb R}
    \{u^\beta:|\beta|_a=D\}.
\end{displaymath}
It is finite-dimensional because $0\leq\beta_i\leq D/a_i$. If no such monomial exists, $V_D=\{0\}$. Define
\begin{displaymath}
    T_x^{(D)}:V_D\longrightarrow\operatorname{Sym}^DW_x^*,
    \qquad R\longmapsto R(P_1(x),\ldots,P_n(x)),
    \qquad K_D(x)=\ker T_x^{(D)}.
\end{displaymath}
Here,
\begin{displaymath}
    R(P_1(x),\ldots,P_n(x))= R(P_1(x;\cdot),\ldots,P_n(x;\cdot)) \in \operatorname{Sym}^DW_x^*.
\end{displaymath}
The domain $V_D$ is the same coefficient space at every base point.

Let $B$ be the set of $x$ for which $P_1(x),\ldots,P_n(x)$ are algebraically dependent.  The preceding decomposition proves both inclusions in
\begin{equation}\label{eq:dependence-countable-union}
    B=\bigcup_{D\geq1}B_D,
    \qquad B_D=\{x\in N:\dim K_D(x)>0\}.
\end{equation}
In fact,
\begin{itemize}
    \item \textbf{$B\subseteq\bigcup_{D\ge 1}B_D$}: If $x\in B$, then there exists a nonzero real-coefficient polynomial $R$ such that
        \begin{displaymath}
            R(P_1(x),\dots,P_n(x))=0.
        \end{displaymath}
        Decompose $R$ by weighted degree as
        \begin{displaymath}
            R=\sum_{D}R_D.
        \end{displaymath}
        By the argument in the preceding subsection, each component $R_D$ also vanishes after substitution. Since $R\neq 0$, at least one $R_D\neq 0$; moreover, since $P_i(x;0)=0$, setting $v=0$ gives $R_0=0$, so this nonzero component satisfies $D\ge 1$. Hence
        \begin{displaymath}
            0\neq R_D\in K_D(x),
        \end{displaymath}
        i.e. $\dim K_D(x)>0$, which means $x\in B_D$.

    \item \textbf{$\bigcup_{D\ge 1}B_D\subseteq B$}: If $x\in B_D$, then $K_D(x)$ contains a nonzero $R$ satisfying
        \begin{displaymath}
            R(P_1(x),\dots,P_n(x))=0.
        \end{displaymath}
        This is already a nonzero polynomial relation among $P_1(x),\dots,P_n(x)$, so $x\in B$.
\end{itemize}
Moreover, formula \eqref{eq:dependence-countable-union} replaces an uncountable collection of possible real coefficients by countably many finite-dimensional kernels. It does not replace real coefficients by rational ones, and it also does not require a uniform bound on relation degrees.

Finally, we have the following equality
\begin{equation}\label{eq:strict-kernel-invariance}
    K_D(\varphi^t x)=K_D(x).
\end{equation}

Indeed, by \eqref{eq:polynomial-equivariance}, for every fixed $R\in V_D$,
\begin{displaymath}
    R(P_1(\varphi^t x),\ldots,P_n(\varphi^t x))\circ C^t(x)
    =
    R(P_1(x),\ldots,P_n(x)).
\end{displaymath}
Here, the composition denotes substitution in the fiber polynomial.

Let $R\in K_D(x)$. Then
\begin{displaymath}
    R(P_1(x),\ldots,P_n(x))=0.
\end{displaymath}
Hence, by the identity above,
\begin{displaymath}
    R(P_1(\varphi^t x),\ldots,P_n(\varphi^t x))\circ C^t(x)=0.
\end{displaymath}
Since $C^t(x)$ is invertible,
\begin{displaymath}
    R(P_1(\varphi^t x),\ldots,P_n(\varphi^t x))=0,
\end{displaymath}
so $R\in K_D(\varphi^t x)$. Thus
\begin{displaymath}
    K_D(x)\subseteq K_D(\varphi^t x).
\end{displaymath}
Applying the same argument with $t$ replaced by $-t$ gives the reverse inclusion
\begin{displaymath}
    K_D(\varphi^t x)\subseteq K_D(x).
\end{displaymath}
Therefore
\begin{displaymath}
    K_D(\varphi^t x)=K_D(x)
\end{displaymath}
as subspaces of the fixed space $V_D$.

\subsubsection{Measurability and a constant relation}
Fix a Euclidean inner product on $V_D$, and choose a continuous fiber metric on $\operatorname{Sym}^D W^*$. In local analytic frames, the matrix entries of $T_x^{(D)}$ are analytic. The sets
\begin{displaymath}
    B_D=\{x\in N:\dim K_D(x)>0\}
\end{displaymath}
are therefore measurable, as can be seen from the vanishing of appropriate minors.

To encode the whole kernel $K_D(x)$ as a measurable matrix-valued function, and to handle rank jumps, we use  the orthogonal projection of $V_{D}$ onto $K_{D}(x)$,
\begin{displaymath}
    \mathsf P_D(x):V_D\longrightarrow V_D.
\end{displaymath}
Since $K_D(x)=\operatorname{im}\mathsf P_D(x)$, the assertion that $K_D(x)$ is almost everywhere constant is equivalent to the assertion that the projection matrix $\mathsf P_D(x)$ is almost everywhere constant. The latter can be treated on an entry-by-entry basis.

The projection admits the regularized limit formula
\begin{equation}\label{eq:projection-formula}
    \mathsf P_D(x)
    =\lim_{r\to\infty}\frac1r
    \left((T_x^{(D)})^*T_x^{(D)}+\frac1rI\right)^{-1}.
\end{equation}
To verify the formula, first note that $\ker T_x^{(D)}=K_D(x)$.

First note that
\begin{displaymath}
    \ker\big((T_x^{(D)})^*T_x^{(D)}\big)
    =\ker T_x^{(D)}
    =K_D(x).
\end{displaymath}
Thus, it suffices to show that the right-hand side of \eqref{eq:projection-formula} fixes every vector in $K_D(x)$ and annihilates every vector in its orthogonal complement.

To prove this, first diagonalize $(T_x^{(D)})^*T_x^{(D)}$.  Let $\tau\ge0$ be an eigenvalue and let $e_\tau$ be a corresponding eigenvector. Then
\begin{displaymath}
    \left((T_x^{(D)})^*T_x^{(D)}+\frac1rI\right)e_\tau
    =\left(\tau+\frac1r\right)e_\tau,
\end{displaymath}
hence
\begin{displaymath}
    \frac1r
    \left((T_x^{(D)})^*T_x^{(D)}+\frac1rI\right)^{-1}
    e_\tau
    =\frac1{r\tau+1}e_\tau.
\end{displaymath}
If $\tau=0$, so that $e_\tau\in K_D(x)$, then $\frac1{r\tau+1}=1$, and the vector is left unchanged in the limit. If $\tau>0$, so that $e_\tau\notin K_D(x)$, then $\frac1{r\tau+1}\to0$, and the vector is annihilated in the limit. Therefore, the limit in \eqref{eq:projection-formula} is the identity on $K_D(x)$ and zero on its orthogonal complement, i.e.\ it is precisely the orthogonal projection onto $K_D(x)$. This also shows that the limit exists and is independent of the chosen frame.

For measurability, note that each finite-$r$ expression
\begin{displaymath}
    \frac1r
    \left((T_x^{(D)})^*T_x^{(D)}+\frac1rI\right)^{-1}
\end{displaymath}
is continuous in local analytic frames: the matrix entries of $T_x^{(D)}$ are analytic and hence continuous, and taking adjoints, products, adding $\frac1rI$, inverting, and multiplying by $\frac1r$ are all continuous operations. Taking the pointwise limit as $r\to\infty$ preserves measurability. Hence $\mathsf P_D(x)$ is a measurable matrix-valued function on $N$. Consequently
\begin{displaymath}
    \dim K_D(x)=\operatorname{tr}\mathsf P_D(x)
\end{displaymath}
is measurable as well, and so $B_D$ is measurable.

Finally, by \eqref{eq:strict-kernel-invariance}, we obtain
\begin{displaymath}
    \mathsf P_D(\varphi^t x)=\mathsf P_D(x),
\end{displaymath}
So $\mathsf {P} _ {D} $ is a flow-invariant matrix-valued function.

Every matrix entry of this projection, written in a fixed basis of $V_D$, is an invariant measurable real function. Ergodicity ensures that all these finitely many entries are constant almost everywhere. Hence, there is a fixed subspace $K_D^0\subset V_D$ such that
\begin{displaymath}
    K_D(x)=K_D^0\quad\text{for }\mu\text{-almost every }x.
\end{displaymath}
There are only countably many $D$, so these statements may be required simultaneously on a common full-measure set.

Suppose $\mu(B)>0$.  By \eqref{eq:dependence-countable-union}, some $B_D$ has positive measure.  The set $B_D$ is invariant, so it has full measure. It follows that $K_D^0\ne\{0\}$.  Choose one fixed $0\ne R\in K_D^0$. The analytic section
\begin{displaymath}
    x\longmapsto R(P_1(x),\ldots,P_n(x))
    \in\operatorname{Sym}^DW_x^*
\end{displaymath}
vanishes almost everywhere. Its zero set is invariant under the flow by \eqref{eq:polynomial-equivariance}. Applying Lemma~\ref{lem:carrier-zero-set}, we obtain
\begin{equation}\label{eq:global-polynomial-relation}
    R(P_1,\ldots,P_n)\equiv0\quad\text{on }N.
\end{equation}
The usual identity theorem alone would not justify this step, since $\mu$ can be singular. It is the minimal-carrier zero-set principle that extends the relation from almost every base point to every base point of $N$.

\subsubsection{A finite decreasing quantity}

For the current list, put
\begin{displaymath}
    \delta=\operatorname{ord}_N\Omega-\sum_{i=1}^na_i+n=w-\sum_{i=1}^na_i+n.
\end{displaymath}
All terms are finite. By \eqref{eq:derivative-order}, each coefficient of $d_PA_i$ has normal order at least $a_i-1$. Expanding the full exterior product gives
\begin{displaymath}
    w\geq\sum_{i=1}^n(a_i-1),
    \qquad \delta\in\mathbb N.
\end{displaymath}
It is essential here that $\Omega$ is the complete ambient form. For example, with $A_1=z$, $A_2=yz$, and $N=\{z=0\}$, one has
\begin{displaymath}
    dA_1\wedge dA_2=z\,dz\wedge dy\not\equiv0.
\end{displaymath}
Its pullback to $N$ is zero, and there is no nonzero all-normal two-form component. Neither of those substitutes would supply a finite termination quantity.

\begin{proposition}\label{prop:finite-elimination}
    Starting with \eqref{eq:current-list-properties}, one can make finitely many constant-coefficient polynomial replacements of the $A_i$ so that these properties continue to hold and the resulting leading normal polynomials $P_1(x),\ldots, P_n(x)$ are algebraically independent on $W_x$ for $\mu$-almost every $x$.
\end{proposition}

\begin{proof}
    If $\mu(B)=0$, there is nothing to do. Otherwise, \eqref{eq:global-polynomial-relation} supplies a nonzero global weighted homogeneous relation. Among all such relations, choose one of the smallest positive weighted degree $D$, and denote it by $R$.

    Since the coefficient field has characteristic zero and $D>0$, there is an index $j$ such that $Q_0:=\partial_jR\ne0$.  Indeed, if all partial derivatives of a real polynomial were zero, every nonconstant monomial would have zero coefficient: differentiating a monomial with $\beta_j>0$ gives the nonzero coefficient $\beta_jc_\beta$, and distinct monomials cannot cancel after differentiating with respect to the same variable. The resulting constant polynomial would be incompatible with a positive weighted degree.

    The polynomial $Q_0$ is weighted homogeneous of degree $D-a_j\geq0$. If this degree is zero, $Q_0$ is a nonzero constant. If it is positive and $Q_0(P_1,\ldots, P_n)$ were identically zero on $N$, it would be a nonzero global relation of smaller degree than $D$. Thus, in either case
    \begin{equation}\label{eq:derivative-relation-nonzero}
        Q_0(P_1,\ldots,P_n)\not\equiv0\quad\text{on }N.
    \end{equation}
    Set
    \begin{displaymath}
        Q=Q_0(A_1,\ldots,A_n),\qquad
        A'_j=R(A_1,\ldots,A_n),\qquad A'_i=A_i\ (i\ne j).
    \end{displaymath}
    Weighted expansion and \eqref{eq:derivative-relation-nonzero} show
    \begin{equation}\label{eq:multiplier-order}
        Q\not\equiv0,
        \qquad \operatorname{ord}_NQ=D-a_j.
    \end{equation}
    For completeness, every monomial of $Q_0$ has weighted degree $D-a_j$. Substituting the leading terms gives its degree-$(D-a_j)$ normal term, namely $Q_0(P)$, while using at least one higher-order remainder raises the normal degree. The displayed nonvanishing gives equality of orders.

    We next estimate the order of $A'_j$ without assuming that its first possible remainder coefficient is nonzero. In adapted coordinates, write $A_i=P_i+\varepsilon_i$, where $\varepsilon_i=O(|z|^{a_i+1})$.  For a monomial $A^\beta$ of $R(A)$, $\sum_i a_i\beta_i=D$.  If $b_i$ of the $\beta_i$ factors of type $i$ are chosen from $\varepsilon_i$, its normal order is at least
    \begin{displaymath}
        \sum_i\left((\beta_i-b_i)a_i+b_i(a_i+1)\right)
        =D+\sum_i b_i.
    \end{displaymath}
    The sum of terms with every $b_i=0$ is $R(P)$, which vanishes identically. Every other term has order at least $D+1$. Hence
    \begin{equation}\label{eq:new-order-lower-bound}
        \operatorname{ord}_NA'_j\geq D+1,
    \end{equation}
    where at this stage the left side could a priori be infinite.

    To rule out that possibility, use the exact ambient chain rule:
    \begin{displaymath}
        d_PA'_j=\sum_{k=1}^n(\partial_kR)(A)\,d_PA_k.
    \end{displaymath}
    In the exterior product with the other $n-1$ differentials in their original positions, all terms except $k=j$ vanish because a factor is repeated. The surviving term is in the original position, so there is no extra sign. Therefore
    \begin{equation}\label{eq:exact-wedge-identity}
        \Omega':=d_PA'_1\wedge\cdots\wedge d_PA'_n=Q\Omega.
    \end{equation}
    Both factors on the right are nonzero analytic objects on the connected set $U$. Their nonzero sets are open and dense, so $Q\Omega\not\equiv0$. In particular, $A'_j\not\equiv0$ and the new list is generically independent in $U$. Its order $a'_j$ is finite by Lemma~\ref{lem:finite-normal-order}, and \eqref{eq:new-order-lower-bound} now gives $a'_j>D$.

    By \eqref{eq:product-order}, \eqref{eq:multiplier-order}, and \eqref{eq:exact-wedge-identity},
    \begin{displaymath}
        w':=\operatorname{ord}_N\Omega'=w+D-a_j.
    \end{displaymath}
    Since $a'_i=a_i$ for $i\ne j$, the new decreasing quantity satisfies
    \begin{displaymath}
        \begin{aligned}
            \delta'
 &=w'-\left(\sum_{i\ne j}a_i+a'_j\right)+n \notag\\
 &=w+D-a_j-\left(\sum_i a_i-a_j+a'_j\right)+n \notag\\
 &=\delta+D-a'_j<\delta.
        \end{aligned}
    \end{displaymath}
    All entries of the new list remain analytic, vanish on $N$, and are constant-coefficient polynomials in the old entries. The Leibniz rule for the Poisson bracket therefore preserves both their mutual commutativity and their commutativity with $H$. The new exterior product is nonzero by \eqref{eq:exact-wedge-identity}, and its new quantity $\delta'$ is again a nonnegative integer.

    Whenever the dependent-fiber set has positive measure, the construction strictly decreases this nonnegative integer. There can be at most the initial value of $\delta$ such replacements. The process must therefore reach a list for which the dependent-fiber set has measure zero. This is a finite analytic construction, with no infinite formal substitution or convergence issue.
\end{proof}

From now on, $A_i,a_i, P_i$ denote the final list and its orders and leading terms. In particular,
\begin{equation}\label{eq:fiber-algebraic-independence}
    P_1(x),\ldots,P_n(x)~\text{are algebraically independent on}~W_x\quad\text{for}~\mu\text{-almost every}~x.
\end{equation}

\subsubsection{From algebraic independence to fiberwise differentials}

\begin{lemma}\label{lem:jacobian-criterion}
    Assume that $P_1,\ldots,P_n\in\mathbb R[z_1,\ldots,z_q]$ are algebraically independent. Then
    \begin{displaymath}
        \{v\in\mathbb R^q:dP_1(v),\ldots,dP_n(v)~\text{are linearly independent}\}
    \end{displaymath}
    is a nonempty open dense set.
\end{lemma}

\begin{proof}
    Let $K=\mathbb R(z_1,\ldots,z_q)$. Extend $P_1,\ldots,P_n$ to a transcendence basis $P_1,\ldots,P_n,T_{n+1},\ldots,T_q$ of $K/\mathbb R$, and put
    \begin{displaymath}
        K_0=\mathbb R(P_1,\ldots,P_n,T_{n+1},\ldots,T_q).
    \end{displaymath}

    The extension $K/K_0$ is algebraic and is generated by the finitely many elements $z_1,\ldots,z_q$, so it is finite. It is separable because the characteristic is zero.

    For each $1\leq j\leq n$, algebraic independence allows us to define an $\mathbb R$-derivation $D_j$ on $K_0$ by
    \begin{displaymath}
        D_jP_i=\delta_{ij}\quad (1\leq i\leq n), \qquad D_jT_k=0\quad (n+1\leq k\leq q).
    \end{displaymath}
    We extend these derivations to $K$ by successively adjoining $z_1,\ldots,z_q$.

    To justify each extension, let $D$ be an $\mathbb R$-derivation on an intermediate field $L$, and let $\alpha$ be algebraic over $L$, with minimal polynomial $M(t)\in L[t]$. For $F(t)\in L[t]$, write $DF$ for the polynomial obtained by applying $D$ to its coefficients, and write $F'$ for its formal derivative with respect to $t$. Since the characteristic is zero, $M'(\alpha)\ne0$. Set
    \begin{displaymath}
        b=-\frac{(DM)(\alpha)}{M'(\alpha)}
    \end{displaymath}
    and define
    \begin{displaymath}
        \widetilde D(F(\alpha))=(DF)(\alpha)+F'(\alpha)b.
    \end{displaymath}
    If $F$ is replaced by $F+MG$, the right-hand side changes by
    \begin{displaymath}
        G(\alpha)\bigl((DM)(\alpha)+M'(\alpha)b\bigr)=0.
    \end{displaymath}
    Thus the definition is independent of the polynomial representative. It is additive, satisfies the Leibniz rule, and agrees with $D$ on $L$. Since $L[\alpha]=L(\alpha)$, it defines a derivation on $L(\alpha)$. Its value on $\alpha$ is forced by differentiating $M(\alpha)=0$, so the extension is unique. Applying this construction successively gives derivations $D_j:K\to K$ with the prescribed values on $P_1,\ldots,P_n$.

    Let $J$ be the $n\times q$ Jacobian matrix
    \begin{displaymath}
        J=\left(\frac{\partial P_i}{\partial z_k}\right),
    \end{displaymath}
    and let $B$ be the $q\times n$ matrix over $K$ defined by
    \begin{displaymath}
        B_{kj}=D_jz_k.
    \end{displaymath}
    Since each $D_j$ annihilates $\mathbb R$, the Leibniz rule applied to the polynomials $P_i$ gives
    \begin{displaymath}
        (JB)_{ij}=\sum_{k=1}^q\frac{\partial P_i}{\partial z_k}\,D_jz_k=D_jP_i=\delta_{ij}.
    \end{displaymath}
    Hence $JB=I_n$, and $J$ has rank $n$ over $K$.

    Consequently, some $n\times n$ minor $\Delta$ of $J$ is a nonzero real polynomial. Its nonzero set
    \begin{displaymath}
        \{v\in\mathbb R^q:\Delta(v)\ne0\}
    \end{displaymath}
    is nonempty, open, and dense. The set on which $J(v)$ has rank $n$ is open, since it is the union of the nonzero sets of its $n\times n$ minors, and it contains the above dense set. Finally, $J(v)$ has rank $n$ exactly when $dP_1(v),\ldots,dP_n(v)$ are linearly independent. This proves the assertion.
\end{proof}

For every $x$ satisfying \eqref{eq:fiber-algebraic-independence}, define
\begin{equation}\label{eq:fiber-regular-set}
    \mathcal O_x=\{v\in W_x: d_vP_1(x;v),\ldots,d_vP_n(x;v)~\text{are linearly independent}\}.
\end{equation}
It is a nonempty open dense subset of $W_x$. The notation $d_v$ always means differentiation in a fixed vector-space fiber. It is different from the ambient differential $d_PA_i$ and from differentiation in the base point $x$.

\subsection{The Poisson structure on the full normal bundle}\label{sec:normal-poisson}

\subsubsection{A constant Poisson structure in each fiber}

Fix $x\in N$, and put $E_x=T_xP$, $T_x=T_xN$.  Let $q_x: E_x\to W_x=E_x/T_x$ be the quotient map. Define
\begin{displaymath}
    \omega_x^\flat:E_x\longrightarrow E_x^*,
    \qquad \omega_x^\flat(u)=\omega_x(u,\cdot),
    \qquad \omega_x^\sharp=(\omega_x^\flat)^{-1}.
\end{displaymath}
The ambient Poisson bilinear form in our convention is
\begin{displaymath}
    J_x(\alpha,\beta)=
    \omega_x(\omega_x^\sharp\alpha,\omega_x^\sharp\beta).
\end{displaymath}
Its projection to the quotient is
\begin{displaymath}
    \Pi_x(\alpha,\beta)
    =J_x(q_x^*\alpha,q_x^*\beta),
    \qquad \alpha,\beta\in W_x^*.
\end{displaymath}
On the vector space $W_x$, this is a constant bivector, and hence defines a Poisson bracket on polynomial functions. The word ``constant'' refers to the fiber variable. The coefficients can and generally do depend analytically on the base point $x$.

To fix the sign in formulas involving the sharp map, we define $\Pi_x^\sharp: W_x^*\to W_x$ by
\begin{displaymath}
    \beta(\Pi_x^\sharp\alpha)=\Pi_x(\beta,\alpha).
\end{displaymath}
With this explicit convention,
\begin{equation}\label{eq:poisson-sharp}
    \Pi_x^\sharp=q_x\omega_x^\sharp q_x^*.
\end{equation}
Indeed, the value of $\beta$ on the right is $\omega_x(\omega_x^\sharp q_x^*\beta, \omega_x^\sharp q_x^*\alpha)=\Pi_x(\beta,\alpha)$. All statements about the radical and isotropic subspaces are, of course, independent of reversing both sharp conventions.

Since the Hamiltonian flow preserves $\omega$ and $TN$, it preserves this quotient construction. For $\alpha,\beta\in W_{\varphi^t x}^*$,
\begin{displaymath}
    \Pi_{\varphi^t x}(\alpha,\beta)
    =\Pi_x\left(C^t(x)^*\alpha,C^t(x)^*\beta\right).
\end{displaymath}

\subsubsection{The leading normal polynomials commute}

In adapted coordinates $(y,z)\in\mathbb R^d\times\mathbb R^q$, we have $N=\{z=0\}$, and each final integral $A_i$ satisfies $A_i|_N\equiv0$ with normal order $a_i\ge1$.  Hence
\begin{displaymath}
    A_i(y,z)=P_i(y;z)+O(|z|^{a_i+1}),
\end{displaymath}
where $P_i(y;z)$ is the leading homogeneous term of degree $a_i$ in $z$.  Differentiating this expansion in the tangential and normal directions gives
\begin{displaymath}
    \partial_zA_i=\partial_zP_i+O(|z|^{a_i}),
    \qquad
    \partial_yA_i=O(|z|^{a_i}).
\end{displaymath}
Thus, the normal derivative retains the leading term $\partial_zP_i$, which is a homogeneous polynomial of degree $a_i-1$ in $z$. In contrast, the tangential derivative does not lower the normal order and remains of order at least $a_i$.

Substituting these derivatives into the ambient Poisson bracket and decomposing the ambient Poisson matrix into $yy$, $yz$, $zy$, and $zz$ blocks, we examine which block can contribute the lowest normal degree. For the $yy$ block, the two factors are $\partial_yA_i=O(|z|^{a_i})$ and $\partial_yA_j=O(|z|^{a_j})$, so this block has normal order at least $a_i+a_j$.  For the $yz$ and $zy$ blocks, one factor is $\partial_yA_i=O(|z|^{a_i})$, while the other is $\partial_zA_j=\partial_zP_j+O(|z|^{a_j})$, whose leading term $\partial_zP_j$ is homogeneous of degree $a_j-1$ in $z$. Hence, these blocks have normal order at least $a_i+a_j-1$. Only the $zz$ block can reach the lower degree: its two factors are $\partial_zP_i$ and $\partial_zP_j$, of degrees $a_i-1$ and $a_j-1$ in $z$, respectively, so the lowest normal degree is
\begin{displaymath}
    (a_i-1)+(a_j-1)=a_i+a_j-2.
\end{displaymath}

Moreover, the $zz$ block of the ambient Poisson matrix has the expansion
\begin{displaymath}
    J_{zz}(y,z)=J_{zz}(y,0)+O(|z|).
\end{displaymath}
Replacing $J_{zz}(y,z)$ by $J_{zz}(y,0)$, or replacing $\partial_zA_i$ and $\partial_zA_j$ by their leading terms $d_zP_i$ and $d_zP_j$, introduces errors of normal degree at least one higher.  Therefore the complete degree-$(a_i+a_j-2)$ part of $\{A_i,A_j\}$ is exactly
\begin{displaymath}
    J_{zz}(y,0)\left(d_zP_i(y;z),\,d_zP_j(y;z)\right).
\end{displaymath}
All other blocks and remainder terms have higher normal degree, so this identification is exact and not merely an approximation to leading order.

We now relate $J_{zz}(y,0)$ to the fiberwise Poisson bivector $\Pi_x$.  In the preceding section, $\Pi_x$ is defined by
\begin{displaymath}
    \Pi_x(\alpha,\beta)=J_x(q_x^*\alpha,q_x^*\beta),
    \qquad \alpha,\beta\in W_x^*,
\end{displaymath}
where $J_x$ is the ambient Poisson bilinear form and $q_x:E_x\to W_x=E_x/T_x$ is the quotient map.  In adapted coordinates, $W_x^*\cong\operatorname{Ann}(T_xN)$, and $J_{zz}(y,0)$ is precisely the matrix representation of $\Pi_x$ in this trivialization.  Hence the degree-$(a_i+a_j-2)$ part above can be written as
\begin{displaymath}
    \Pi_x\left(d_{v}P_i(x;v),\,d_{v}P_j(x;v)\right),
\end{displaymath}
where $d_{v}P_i(x;v)\in W_x^*$ is the differential of $P_i(x;\cdot)$ at $v\in W_x$.

We now use the commutation assumption $\{A_i, A_j\}=0$. Since the Poisson bracket vanishes identically, each component of its normal expansion vanishes. Therefore, every normal-degree component on both sides must vanish. In particular, its degree-$(a_i+a_j-2)$ part must vanish, which yields
\begin{equation}\label{eq:leading-poisson-commutation}
    \Pi_x\left(d_vP_i(x;v),d_vP_j(x;v)\right)=0\qquad(x\in N,\ v\in W_x).
\end{equation}
Equation~\eqref{eq:leading-poisson-commutation} says that the fiber differentials $d_vP_i(x;v)$ span an isotropic subspace of $(W_x^*,\Pi_x)$ under the fiberwise Poisson bivector $\Pi_x$: for every base point $x$ and every normal vector $v$, the differentials $d_{v}P_i(x;v)$ and $d_{v}P_j(x;v)$ are pairwise orthogonal with respect to $\Pi_x$.

The entire comparison takes place in the ambient neighborhood $U$. On the regular energy hypersurface $\Sigma$, the restricted form is degenerate, with
\begin{displaymath}
    \ker(\omega_{x}|_{T_{x}\Sigma})=\mathbb{R}X_{H}(x).
\end{displaymath}
Moreover, if one first restricts the original integrals to $\Sigma$ and then uses the normal bundle $T\Sigma|_{N}/TN$, the original energy integral $H-e$ is identically zero and has infinite normal order. The ambient construction uses the nondegenerate form $\omega$ on $TP|_{N}$ and retains the finite normal order of $H-e$ at the start of the elimination argument.

\subsubsection{The radical and its dimension}

Write $\operatorname{rank}(\omega_x|_{T_x})=2s$.  The annihilator identification and the symplectic orthogonal are
\begin{displaymath}
    q_x^*:W_x^*\xrightarrow{\ \simeq\ }\operatorname{Ann}(T_x),
    \qquad
    \omega_x^\sharp\operatorname{Ann}(T_x)=T_x^\omega.
\end{displaymath}
Under their composition, $\Pi_x$ is identified with $\omega_x|_{T_x^\omega}$.  Indeed, by the definition of $\Pi_x$,
\begin{displaymath}
    \Pi_x(\alpha,\beta)
    =J_x(q_x^*\alpha,q_x^*\beta)
    =\omega_x\left(\omega_x^\sharp q_x^*\alpha,
    \omega_x^\sharp q_x^*\beta\right),
\end{displaymath}
and since $\omega_x^\sharp\operatorname{Ann}(T_x)=T_x^\omega$, both vectors $\omega_x^\sharp q_x^*\alpha$ and $\omega_x^\sharp q_x^*\beta$ lie in $T_x^\omega$.  Hence $\Pi_x$ agrees with $\omega_x|_{T_x^\omega}$ under the natural identification.

The radical of the latter is
\begin{displaymath}
    \operatorname{rad} \left(\omega_x|_{T_x^\omega}\right)
    =\{\xi\in T_x^\omega:\omega_x(\xi,\eta)=0,
    \ \forall\eta\in T_x^\omega\}
    =T_x^\omega\cap(T_x^\omega)^\omega
    =T_x^\omega\cap T_x,
\end{displaymath}
where $(T_x^\omega)^\omega=T_x$ follows from the nondegeneracy of $\omega_x$.  This is also the radical of $\omega_x|_{T_x}$. Since $\operatorname{rank}(\omega_x|_{T_x})=2s$, the Rank-Nullity Theorem implies
\begin{displaymath}
    \dim(T_x^\omega\cap T_x)
    =\dim T_x-\operatorname{rank}(\omega_x|_{T_x})
    =d-2s.
\end{displaymath}
Moreover, $\dim T_x^\omega=2n-d$.  Therefore, on setting
\begin{displaymath}
    \mathcal C_x=\ker\Pi_x^\sharp,
    \qquad k=\dim\mathcal C_x,
    \qquad 2\ell=\operatorname{rank}\Pi_x,
\end{displaymath}
we obtain the pointwise identities
\begin{equation}\label{eq:normal-poisson-dimensions}
    k=d-2s,
    \qquad 2\ell=2n-d-k=2n-2d+2s,
    \qquad k+\ell=n-s.
\end{equation}
These formulas do not require the rank of $\omega|_{TN}$ to be constant in advance.

\begin{lemma}\label{lem:carrier-isotropic}
    The restriction of $\omega$ to $TN$ vanishes identically. In particular, $d\leq n$, and on all of $N$,
    \begin{displaymath}
        \dim\mathcal C_x=d,
        \qquad \operatorname{rank}\Pi_x=2(n-d).
    \end{displaymath}
\end{lemma}

\begin{proof}
    For an alternating form of rank $2\ell$ and radical dimension $k$, an isotropic subspace $L$ has dimension at most $k+\ell$. Indeed, the image of $L$ in the quotient by the radical is isotropic in a symplectic space of dimension $2\ell$, so its dimension is at most $\ell$. The kernel of that projection on $L$ has dimension at most $k$.

    For almost every $x$, take any $v\in\mathcal O_x$ and set
    \begin{equation}\label{eq:gradient-isotropic-subspace}
        L_v=\operatorname{span}_{\mathbb R}
        \{d_vP_1(x;v),\ldots,d_vP_n(x;v)\}\subset W_x^*.
    \end{equation}
    By \eqref{eq:fiber-regular-set}, $\dim L_v=n$.  By \eqref{eq:leading-poisson-commutation}, it is isotropic for $\Pi_x$. Using \eqref{eq:normal-poisson-dimensions}, we get
    \begin{displaymath}
        n\leq k+\ell=n-s.
    \end{displaymath}
    Thus $s=0$ for almost every $x$.

    Consider the analytic section
    \begin{displaymath}
        S(x):=\omega_x|_{T_xN}\in\Lambda^2T_x^*N,
        \qquad x\in N,
    \end{displaymath}
    so that $S\in\Gamma^\omega(N,\Lambda^2T^*N)$.  For every $x\in N$ and every $u,w\in T_xN$, the symplectic invariance of $\omega$ and the invariance $\varphi^t(N)=N$ give
    \begin{displaymath}
        S(\varphi^t x)(A^{t}(x)u,A^{t}(x)w)
        =\omega_{\varphi^t x}(A^{t}(x)u,A^{t}(x)w)
        =\omega_x(u,w)
        =S(x)(u,w),
    \end{displaymath}
    where $A^{t}(x)=D\varphi^t_x$.  Hence $S(\varphi^t x)=0$ if and only if $S(x)=0$, so the zero set
    \begin{displaymath}
        Z_S=\{x\in N:S(x)=0\}
    \end{displaymath}
    is invariant under the flow. Since $s=0$ for almost every $x$, we have $\mu(Z_S)=1$.  Lemma~\ref{lem:carrier-zero-set} now gives
    \begin{equation}\label{eq:carrier-isotropic}
        \omega|_{TN}\equiv0\quad\text{on }N.
    \end{equation}

    It remains to deduce the stated dimensions. From \eqref{eq:carrier-isotropic} we have $T_xN\subseteq T_xN^\omega$ for every $x\in N$, so
    \begin{displaymath}
        \dim(T_xN\cap T_xN^\omega)=\dim T_xN=d.
    \end{displaymath}
    Combining this with the identity $k=d-2s$ in \eqref{eq:normal-poisson-dimensions} and with $s=0$ yields
    \begin{displaymath}
        k=d,\qquad 2\ell=2n-2d,\qquad \ell=n-d.
    \end{displaymath}
    Thus
    \begin{displaymath}
        \dim\mathcal C_x=d,
        \qquad \operatorname{rank}\Pi_x=2(n-d).
    \end{displaymath}
    Finally, since $T_xN$ is isotropic for the nondegenerate form $\omega_x$ on $T_xP$, we have $\dim T_xN\leq\tfrac12\dim T_xP=n$, i.e.\ $d\leq n$.
\end{proof}

The carrier is therefore isotropic. We next control its tangent Lyapunov exponents. The dimension bound has shown that the carrier is isotropic. We next use the invariance of the normal polynomials to control the Lyapunov exponents tangent to this carrier.

\subsection{Invariant tensors and zero tangent exponents}\label{sec:zero-exponents}

\subsubsection{Log-integrability of the induced cocycles}

Fix a smooth Riemannian metric in an ambient neighborhood of $\Sigma$. Use its norm on $TP|_\Sigma$, the restricted norm on $TN$, and the quotient norm on $W$. Write
\begin{displaymath}
    A_x=D\varphi^1_x:T_xP\longrightarrow T_{fx}P,
    \qquad C_x=C^1(x):W_x\longrightarrow W_{fx}.
\end{displaymath}
The ambient maps $A_x$ and $A_x^{-1}$ depend continuously on $x\in\Sigma$. By compactness, there is $B\geq1$ such that
\begin{displaymath}
    \sup_{x\in\Sigma}\max\{\|A_x\|,\|A_x^{-1}\|\}\leq B.
\end{displaymath}
The restrictions to $TN$ and their inverses satisfy the same bound. Since $A_x(T_xN)=T_{fx}N$, for every quotient vector $[v]\in W_x$,
\begin{align*}
    \|C_x[v]\|&=\inf_{w'\in T_{fx}N}\|A_xv+w'\|\\
              &=\inf_{w\in T_xN}\|A_x(v+w)\|
              \leq B\inf_{w\in T_xN}\|v+w\|=B\|[v]\|.
\end{align*}
The same calculation for the inverse gives
\begin{displaymath}
    \|C_x\|\leq B,\qquad \|C_x^{-1}\|\leq B.
\end{displaymath}

The forward cocycle on the dual bundle is
\begin{equation}\label{eq:dual-cocycle}
    \widehat C_x=(C_x^{-1})^*:W_x^*\longrightarrow W_{fx}^*.
\end{equation}
Its direction is chosen so that $\widehat {C} _ {x} $ is a forward cocycle on $W^*$. Since the adjoint of a linear operator has the same norm,
\begin{displaymath}
    \|\widehat C_x\|=\|(C_x^{-1})^*\|=\|C_x^{-1}\|\leq B,
    \qquad
    \|\widehat C_x^{-1}\|=\|C_x^*\|=\|C_x\|\leq B.
\end{displaymath}
Their tensor powers of any fixed degree $a$ have norm at most $B^a$, and the same is true after restriction to the symmetric tensors with the induced norm. Indeed, fix Hilbert tensor norms on $W_x^{\otimes a}$ and on $(W_x^*)^{\otimes a}$ induced by the chosen inner products on $W_x$ and $W_x^*$, respectively. Let $T$ be a linear operator with $\|T\|\leq B$.  By the singular value decomposition, the singular values of $T^{\otimes a}$ are the $a$-fold products of the singular values of $T$; hence
\begin{displaymath}
    \|T^{\otimes a}\|=\|T\|^a\leq B^a.
\end{displaymath}
The symmetric power is the restriction of $T^{\otimes a}$ to the symmetric subspace, so
\begin{displaymath}
    \|\operatorname{Sym}^a T\|\leq\|T\|^a\leq B^a.
\end{displaymath}
The same argument applies to the inverses.

Thus, every cocycle $L$ used below satisfies the two integrability conditions
\begin{equation}\label{eq:oseledets-integrability}
    \int_N\log^+\|L_x\|\,d\mu(x)<\infty,\qquad\int_N\log^+\|L_x^{-1}\|\,d\mu(x)<\infty.
\end{equation}
Indeed, since $\|L_x\|,\|L_x^{-1}\|\leq B_a$ uniformly in $x$ for some constant $B_a$ depending only on $a$, we have
\begin{displaymath}
    \log^+\|L_x\|\leq\log^+B_a<\infty,\qquad\log^+\|L_x^{-1}\|\leq\log^+B_a<\infty,
\end{displaymath}
and $\mu$ is a probability measure, so both integrals are finite.

The invertible multiplicative ergodic theorem applies to the probability space $(N,\mu,f)$ \cite{Oseledets1968, Filip2019}. The required bounds follow from compactness of the energy hypersurface and hold even when \(N\) is noncompact.

\subsubsection{An invariant measurable section has only zero exponents}

\begin{lemma}\label{lem:invariant-section-zero-exponent}
    Let $L$ be an invertible finite-dimensional measurable linear cocycle over an invertible ergodic probability-preserving transformation, satisfying \eqref{eq:oseledets-integrability}. Suppose a measurable section $s$ with $\parallel s(x)\parallel < \infty$ almost everywhere satisfies
    \begin{displaymath}
        s(fx)=L_xs(x)
    \end{displaymath}
    almost everywhere. Then $s(x)$ belongs to the zero Lyapunov subspace almost everywhere. Neither integrability nor boundedness of $\|s\|$ is required.
\end{lemma}

\begin{proof}
    By \eqref{eq:oseledets-integrability}, the invertible multiplicative ergodic theorem applies to the measurable vector bundle $\mathcal E$ over the probability-preserving system $(N,\mu,f)$. It yields a full-measure $f$-invariant Oseledets regular set $\mathcal R$ on which
    \begin{displaymath}
        \mathcal E_x=\bigoplus_\rho \mathcal E_x^\rho,
    \end{displaymath}
    where $\rho$ ranges over finitely many Lyapunov exponents. Each subspace $\mathcal {E} _ {x} ^\rho$ is invariant along the orbit,
    \begin{displaymath}
        L_x\mathcal E_x^\rho=\mathcal E_{fx}^\rho.
    \end{displaymath}
    Let
    \begin{displaymath}
        L_x:\mathcal E_x\longrightarrow \mathcal E_{fx}
    \end{displaymath}
    be the cocycle. We work on a common invariant set of full measure on which the Oseledets splitting and the equivariance identity for $s$ hold along the entire two-sided orbit. If the identity $s(fx)=L_xs(x)$ holds only on a full-measure set $E$, we replace $E$ by
    \begin{displaymath}
        \bigcap_{k\in\mathbb Z}f^kE
    \end{displaymath}
    and intersect it with the Oseledets regular set; this remains a full-measure set, and all iterated identities below are then valid simultaneously. For $m\ge 1$, define
    \begin{displaymath}
        L_x^{(m)}=L_{f^{m-1}x}\cdots L_x,
        \qquad
        L_x^{(-m)}=\left(L_{f^{-m}x}^{(m)}\right)^{-1}.
    \end{displaymath}
    Then, for every nonzero $v\in \mathcal E_x^\rho$,
    \begin{displaymath}
        \lim_{m\to\infty}\frac1m\log\|L_x^{(m)}v\|=\rho,
        \qquad
        \lim_{m\to\infty}\frac1m\log\|L_x^{(-m)}v\|=-\rho.
    \end{displaymath}
    Since $f$ is ergodic, the Lyapunov exponents are $\mu$-almost everywhere constant, and there are only finitely many distinct values.

    For each $\rho$, let $\pi_\rho(x):\mathcal E_x\to \mathcal E_x$ be the projection onto $\mathcal{E}_x^\rho$ along the other Oseledets subspaces.  It is linear, idempotent, and satisfies
    \begin{displaymath}
        \operatorname{im}\pi_\rho(x)=\mathcal E_x^\rho,
        \qquad
        \ker\pi_\rho(x)=\bigoplus_{\rho'\neq\rho}\mathcal E_x^{\rho'},
        \qquad
        \sum_\rho\pi_\rho(x)=\mathrm{id}_{\mathcal E_x}.
    \end{displaymath}
    Moreover, it is equivariant:
    \begin{displaymath}
        \pi_\rho(fx)\circ L_x=L_x\circ\pi_\rho(x).
    \end{displaymath}
    Consequently, if $s(fx)=L_xs(x)$ and we set $s_\rho(x)=\pi_\rho(x)s(x)$, then
    \begin{displaymath}
        s_\rho(fx)
        =\pi_\rho(fx)s(fx)
        =\pi_\rho(fx)L_xs(x)
        =L_x\pi_\rho(x)s(x)
        =L_xs_\rho(x),
    \end{displaymath}
    so $s_\rho$ is again an invariant section. It is measurable and has finite norm almost everywhere, because $\pi_\rho$ is measurable and $s$ has finite norm almost everywhere.

    The projection $\pi_\rho(x)$ need not have uniformly bounded norm: its norm may grow with $x$, since $\mathcal{E}_x^\rho$ and the sum of the other Oseledets subspaces may become nearly parallel in the original norm. Thus we cannot assert that $\|s_\rho\|$ is bounded or integrable. This is precisely why the present lemma assumes only that $s$ has finite norm almost everywhere.

    For each positive integer $r$, set
    \begin{displaymath}
        B_{r,\rho}=\{x:\|s_\rho(x)\|\le r\}.
    \end{displaymath}
    These sets cover the Oseledets set as $r$ varies. By the Poincar\'e recurrence theorem, for each $B_{r,\rho}$ and for both $f$ and $f^{-1}$, almost every point of $B_{r,\rho}$ returns to $B_{r,\rho}$ infinitely often in both time directions.  There are countably many sets $B_{r,\rho}$, so the union of the resulting null exceptional sets is again null. For every remaining $x$, choose $r$ with $x\in B_{r,\rho}$; then there are arbitrarily large positive and negative return times $m_j$ with $f^{m_j}x\in B_{r,\rho}$.

    If $\rho>0$ and $s_\rho(x)\ne0$, the multiplicative ergodic theorem gives
    \begin{displaymath}
        \lim_{m\to\infty}\frac1m
        \log\|L_x^{(m)}s_\rho(x)\|=\rho>0.
    \end{displaymath}
    At positive return times $m_j\to\infty$, however,
    \begin{displaymath}
        \|L_x^{(m_j)}s_\rho(x)\|
        =\|s_\rho(f^{m_j}x)\|\le r,
    \end{displaymath}
    which is a contradiction. If $\rho<0$, apply the same argument in negative time: the backward growth exponent is $-\rho>0$, again contradicting the bounded values at return times. Every component with a nonzero exponent vanishes, as claimed.
\end{proof}

\subsubsection{Zero total weights of the leading polynomials}

At an Oseledets-regular point for the quotient cocycle, write
\begin{equation}\label{eq:quotient-oseledets-splitting}
    W_x=\bigoplus_\lambda W_x^\lambda.
\end{equation}
Here $\lambda$ ranges over the finitely many Lyapunov exponents of the quotient cocycle $C_x$, and each subspace $W_x^\lambda$ is invariant along the orbit:
\begin{displaymath}
    C_xW_x^\lambda=W_{fx}^\lambda.
\end{displaymath}
By the Oseledets theorem, $C_x$ has exponent $\lambda$ on $W_x^\lambda$, i.e.\ for every nonzero $v\in W_x^\lambda$,
\begin{displaymath}
    \lim_{m\to\pm\infty}\frac1m\log\|C_x^{(m)}v\|=\lambda.
\end{displaymath}
The dual cocycle \eqref{eq:dual-cocycle} is
\begin{displaymath}
    \widehat C_x=(C_x^{-1})^*:W_x^*\longrightarrow W_{fx}^*.
\end{displaymath}
It is a forward cocycle on the dual bundle $W^*$. By the Oseledets theorem, the forward cocycle $C$ has exponent $\lambda$ on $W_x^\lambda$, and the dual cocycle $\widehat C$ has exponent $-\lambda$ on the annihilator
\begin{displaymath}
    (W_x^\lambda)^*
    :=\operatorname{Ann}_{W_x^*}\!\left(
    \bigoplus_{\lambda'\ne\lambda} W_x^{\lambda'}\right).
\end{displaymath}
The iterations of the forward contragredient cocycle satisfy
\begin{displaymath}
    \widehat C_x^{(m)}=\left((C_x^{(m)})^{-1}\right)^*.
\end{displaymath}
Therefore, the dual Oseledets rule gives exponent $-\lambda$ on
\begin{displaymath}
    (W_x^\lambda)^*
    :=\operatorname{Ann}_{W_x^*}\!\left(
    \bigoplus_{\lambda'\ne\lambda} W_x^{\lambda'}\right).
\end{displaymath}
Here $(W_x^\lambda)^*$ denotes the corresponding subspace of $W_x^*$, identified with the dual of $W_x^\lambda$ by restriction. See \cite[Section~3.1.1]{Filip2019} for the dual and tensor rules used here. Here $(W_x^\lambda)^*$ is regarded as the subspace of $W_x^*$ that annihilates all the other Oseledets subspaces, so that
\begin{displaymath}
    W_x^*=\bigoplus_\lambda (W_x^\lambda)^*,
\end{displaymath}
with exponent $-\lambda$ on the summand $(W_x^\lambda)^*$.

Tensor-product exponents add, and the same is true for symmetric powers; see \cite[Section~3.1.1]{Filip2019}. One can also obtain this rule directly from the two-sided Oseledets estimates in a Lyapunov norm. On a tensor product of subspaces with exponents $\rho_1,\ldots,\rho_a$, multiplying the upper and lower bounds yields the exponent $\rho_1+\cdots+\rho_a$. The measurable changes of norm have subexponential growth. Passing to the symmetric subspace preserves the decomposition by total weight, i.e.\ the induced exponent on each symmetric tensor of a fixed total weight is the sum of the exponents of its factors.

Formula~\eqref{eq:polynomial-equivariance} at $t=1$ is equivalent to
\begin{equation}\label{eq:invariant-polynomial-tensor}
    P_i(fx)=\operatorname{Sym}^{a_i}(\widehat C_x)P_i(x).
\end{equation}
In fact, the equivariance identity reads $P_i(fx;C_xv)=P_i(x;v)$ for every $v\in W_x$.  Writing $\alpha=C_xv\in W_{fx}$, so that $v=C_x^{-1}\alpha$, we obtain
\begin{displaymath}
    P_i(fx;\alpha)
    =P_i(x;C_x^{-1}\alpha)
    =\Bigl(\operatorname{Sym}^{a_i}(C_x^{-1})^*P_i(x)\Bigr)(\alpha)
    =\Bigl(\operatorname{Sym}^{a_i}(\widehat C_x)P_i(x)\Bigr)(\alpha).
\end{displaymath}
Here $\operatorname{Sym}^{a_i}(\widehat C_x)$ denotes the induced action of $\widehat C_x$ on the symmetric power $\operatorname{Sym}^{a_i}W_x^*$, which is the natural space in which $P_i(x)$ lies.  Since both $P_i(fx)$ and $P_i(x)$ are homogeneous polynomials of degree $a_i$, this yields \eqref{eq:invariant-polynomial-tensor}.

Each $P_i$ is a measurable invariant section for this symmetric-power cocycle, and has finite norm at every base point. Indeed, by Section~\ref{subsec:1}, $P_i$ is an analytic section of $\operatorname{Sym}^{a_i}W^*$, hence has finite norm at every base point and is measurable in local analytic frames; and \eqref{eq:invariant-polynomial-tensor} is precisely the invariance required by Lemma~\ref{lem:invariant-section-zero-exponent}. That lemma therefore says that $P_i$ has only zero total tensor weight.

To record the resulting polynomial identity, define the linear map $D_x: W_x\to W_x$ by
\begin{displaymath}
    D_x|_{W_x^\lambda}=\lambda\,\mathrm{id}.
\end{displaymath}
Choose a linear basis adapted to \eqref{eq:quotient-oseledets-splitting}, and let $\lambda_1,\ldots,\lambda_q$ be the exponents associated to its coordinate vectors. A monomial $z^\beta$ represents a dual symmetric tensor of exponent $-\sum_j\beta_j\lambda_j$. Since $P_i$ has only zero total tensor weight, every nonzero monomial of $P_i(x;z)$ satisfies
\begin{displaymath}
    \sum_j\beta_j\lambda_j=0.
\end{displaymath}
Equivalently,
\begin{displaymath}
    P_i(x;e^{tD_x}v)=P_i(x;v)
    \quad(t\in\mathbb R,\ v\in W_x).
\end{displaymath}
In fact, $e^{tD_x}$ acts on $W_x^\lambda$ by multiplication by $e^{t\lambda}$, so a monomial $z^\beta$, regarded as a polynomial on $W_x$, transforms as
\begin{displaymath}
    z^\beta(e^{tD_x}v)
    =e^{t\sum_j\beta_j\lambda_j}z^\beta(v),
\end{displaymath}
and the exponent vanishes by the identity above. Differentiating in $t$ at zero gives
\begin{equation}\label{eq:weight-differential-identity}
    d_vP_i(x;v)(D_xv)=0
    \quad(v\in W_x,\ 1\leq i\leq n)
\end{equation}
for almost every $x$.

A total weight of zero does not mean that $P_i$ depends only on zero-exponent variables. A product of variables with exponents $\lambda$ and $-\lambda$ also has total weight zero. Identity \eqref{eq:weight-differential-identity}, which includes such cancellations, is the precise consequence that will be used.

\subsubsection{The Poisson radical lies in the zero dual exponent space}

By Lemma~\ref{lem:carrier-isotropic}, $\mathcal C_x=\ker\Pi_x^\sharp$ has dimension $d$, and the symplectic quotient $W_x^*/\mathcal C_x$ has dimension $2(n-d)$.  Fix a point $x$ at which both fiberwise independence and all conclusions of Oseledec's theorem hold.

For $v\in\mathcal O_x$, let $L_v$ be the subspace from \eqref{eq:gradient-isotropic-subspace}.  By \eqref{eq:fiber-regular-set}, $\dim L_v=n$.  By \eqref{eq:leading-poisson-commutation}, $L_v$ is isotropic for $\Pi_x$.  Its image in the symplectic quotient $W_x^*/\mathcal C_x$ has dimension at most $n-d$. Indeed, $W_x^*/\mathcal C_x$ has dimension $2(n-d)$ and carries the nondegenerate alternating form induced by $\Pi_x$, and an isotropic subspace of a nondegenerate symplectic space has dimension at most half of the ambient dimension. Therefore
\begin{displaymath}
    \dim(L_v\cap\mathcal C_x)
    =\dim L_v-\dim\left((L_v+\mathcal C_x)/\mathcal C_x\right)
    \geq n-(n-d)=d,
\end{displaymath}
since $(L_v+\mathcal C_x)/\mathcal C_x$ is precisely the image of $L_v$ in $W_x^*/\mathcal C_x$.  Since $\dim\mathcal C_x=d$ and $L_v\cap\mathcal C_x\subseteq\mathcal C_x$, we obtain $\dim(L_v\cap\mathcal C_x)=d$, which proves
\begin{equation}\label{eq:radical-in-gradients}
    \mathcal C_x\subset L_v\quad(v\in\mathcal O_x).
\end{equation}

Every covector in $L_v$ annihilates $D_xv$.  Indeed, by \eqref{eq:weight-differential-identity},
\begin{displaymath}
    d_vP_i(x;v)(D_xv)=0,\qquad 1\leq i\leq n,
\end{displaymath}
and $L_v$ is spanned by $d_vP_1(x;v),\ldots,d_vP_n(x;v)$, so the vanishing extends to every element of $L_v$ by linearity.  Now fix $c\in\mathcal C_x$.  Using \eqref{eq:radical-in-gradients}, we have $c\in L_v$, hence
\begin{displaymath}
    c(D_xv)=0\quad(v\in\mathcal O_x).
\end{displaymath}
The left side is a linear function of $v$. Since $\mathcal {O} _ {x} $ is dense in $ W _ {x} $, it vanishes for all $v\in W _ {x} $.  Thus $c$ annihilates
\begin{displaymath}
    \operatorname{im}D_x=\bigoplus_{\lambda\ne0}W_x^\lambda.
\end{displaymath}
In fact, $D_x$ acts as multiplication by $\lambda$ on each $W_x^\lambda$, so it vanishes on $W_x^0$ and is invertible on each $W_x^\lambda$ with $\lambda\ne0$; consequently its image is exactly $\bigoplus_{\lambda\ne0}W_x^\lambda$.  The annihilator of this sum is exactly the zero-exponent subspace for the dual cocycle, i.e.
\begin{displaymath}
    \operatorname{Ann}_{W_x^*}\!\left(
    \bigoplus_{\lambda\ne0}W_x^\lambda\right)
    =(W_x^*)^0.
\end{displaymath}
Equivalently, in terms of the dual of the quotient,
\begin{displaymath}
    \operatorname{Ann}_{W_x^*}\!\left(
    \bigoplus_{\lambda\ne0}W_x^\lambda\right)
    \cong
    \left(W_x\Big/\bigoplus_{\lambda\ne0}W_x^\lambda\right)^*.
\end{displaymath}
Hence
\begin{equation}\label{eq:radical-zero-exponent}
    \mathcal C_x\subset (W_x^*)^0
    \quad\text{for }\mu\text{-almost every }x.
\end{equation}
Here, the superscript $0$ denotes the Oseledets zero exponent, not the zero vector space. Notice that no assertion that $D_x$ preserves the Poisson form was needed: the argument uses only the definition of $D_x$ and the vanishing \eqref{eq:weight-differential-identity}.

\subsubsection{A bounded equivariant identification with the tangent bundle}
For $u\in T_xN$, define
\begin{displaymath}
    b_x(u)([v])=\omega_x(u,v),\qquad [v]\in W_x.
\end{displaymath}
This is well-defined because $N$ is isotropic: if $[v]=[v']$, then $v'=v+w$ for some $w\in T_xN$, and
\begin{displaymath}
    b_x(u)([v'])=\omega_x(u,v)+\omega_x(u,w)=\omega_x(u,v)=b_x(u)([v]).
\end{displaymath}
Linearity in $u$ is immediate from the bilinearity of $\omega_x$.

In terms of the annihilator identification,
\begin{equation}\label{eq:pullback-tangent-map}
    q_x^*b_x(u)=\omega_x^\flat u.
\end{equation}
Indeed, both sides are elements of $T_x^*P$ whose value on any $v\in T_xP$ equals $\omega_x(u,v)$.

Using the sharp convention \eqref{eq:poisson-sharp}, we compute
\begin{displaymath}
    \Pi_x^\sharp b_x(u)
    =q_x\omega_x^\sharp q_x^*b_x(u)
    =q_x\omega_x^\sharp\omega_x^\flat u
    =q_x(u)=0.
\end{displaymath}
Here we used \eqref{eq:pullback-tangent-map} and the identity $\omega_x^\sharp\omega_x^\flat u=u$, together with the fact that the quotient map $q_x$ vanishes on $T_xN$.  Thus $b_x(T_xN)\subset\mathcal C_x$.  If $b_x(u)=0$, then $\omega_x(u,v)=0$ for every $v\in T_xP$, and nondegeneracy gives $u=0$; so $b_x$ is injective.  By Lemma~\ref{lem:carrier-isotropic}, $\dim\mathcal C_x=d$, and $\dim T_xN=d$.  Since $b_x$ is injective and both spaces have dimension $d$, it is also surjective, hence
\begin{equation}\label{eq:tangent-radical-isomorphism}
    b_x:T_xN\xrightarrow{\ \simeq\ }\mathcal C_x.
\end{equation}

This isomorphism is strictly equivariant. For $u\in T_xN$ and $[v]\in W_x$, symplectic invariance gives
\begin{align*}
    b_{fx}(A_xu)(C_x[v])
 &=\omega_{fx}(A_xu,A_xv)\\
 &=\omega_x(u,v)=b_x(u)([v]).
\end{align*}
Since $[v]\in W_x$ is arbitrary, the identity above says that the two linear forms on $W_{fx}$ coincide on the image of $C_x$.  As $C_x$ is surjective, this implies
\begin{displaymath}
    b_{fx}(A_xu)=\widehat C_x b_x(u)
\end{displaymath}
as elements of $W_{fx}^*$.  Hence, since $u\in T_xN$ is also arbitrary,
\begin{equation}\label{eq:bounded-conjugacy-equivariance}
    b_{fx}\circ A_x|_{T_xN}=\widehat C_x\circ b_x.
\end{equation}
The forward dual map is $(C_x^{-1})^*$, as in \eqref{eq:dual-cocycle}; the unmodified pullback $C_x^*$ goes in the opposite direction.

We also need uniform bounds, rather than just a pointwise measurable isomorphism. On the compact set $\Sigma$, $\omega^\flat$ and its inverse are uniformly bounded. Indeed, both are continuous in $x$, and a continuous family of linear isomorphisms on a compact set has uniformly bounded norm and inverse norm. Hence there are constants $0<c\leq C<\infty$ such that
\begin{equation}\label{eq:uniform-symplectic-isomorphism}
    c\|u\|\leq\|\omega_x^\flat u\|\leq C\|u\|
    \quad(x\in\Sigma,\ u\in T_xP).
\end{equation}
For an inner-product space $E$ and a subspace $T$, the quotient norm identifies $E/T$ isometrically with $T^\perp$, since the quotient norm of a class equals the norm of its orthogonal projection onto $T^\perp$. Its dual is therefore isometric to $\operatorname{Ann}(T)\subset E^*$. Applied here, this shows that $q_x^{*}:W_x^*\to\operatorname{Ann}(T_xN)$ is an isometry. Combining \eqref{eq:pullback-tangent-map} with \eqref{eq:uniform-symplectic-isomorphism}, we obtain
\begin{equation}\label{eq:uniform-b-bounds}
    c\|u\|\leq\|b_x(u)\|\leq C\|u\|\quad (x\in N,u\in T_xN).
\end{equation}

The constants do not depend on $x$. This uses the full nondegenerate symplectic form on $TP|_\Sigma$ and the quotient-dual norm; it does not require compactness of $N$ or a continuous extension of $TN$ to $\overline N$.

\begin{proposition}\label{prop:tangent-zero-exponents}
    For $\mu$-almost every $x\in N$ and every nonzero $u\in T_xN$,
    \begin{equation}\label{eq:tangent-zero-exponents}
        \lim_{m\to+\infty}\frac1m\log\|D\varphi^m_xu\|=0,
        \qquad
        \lim_{m\to+\infty}\frac1m\log\|D\varphi^{-m}_xu\|=0.
    \end{equation}
\end{proposition}

\begin{proof}
    Choose a common invariant full-measure set for the preceding conclusions. For $x$ in that set, \eqref{eq:radical-zero-exponent} and \eqref{eq:tangent-radical-isomorphism} imply $0\ne b_x(u)\in(W_x^*)^0$ whenever $0\ne u\in T_xN$.  Indeed, \eqref{eq:tangent-radical-isomorphism} gives $b_x(u)\in\mathcal C_x$ and $b_x(u)\ne0$, while \eqref{eq:radical-zero-exponent} gives $\mathcal C_x\subset(W_x^*)^0$.

    Iterating \eqref{eq:bounded-conjugacy-equivariance}, for every integer $m$,
    \begin{displaymath}
        b_{f^m x}(D\varphi^m_xu)=\widehat C_x^{(m)}b_x(u).
    \end{displaymath}
    This is obtained by applying the equivariance identity along the orbit of $x$ and composing the resulting expressions for $m=1,2,\ldots$ and $m=-1,-2,\ldots$.

    The two norms in this identity are uniformly comparable by \eqref{eq:uniform-b-bounds}: there exist constants $c,C>0$ independent of $x,u,m$ such that
    \begin{displaymath}
        c\|D\varphi^m_xu\|
        \leq\|\widehat C_x^{(m)}b_x(u)\|
        \leq C\|D\varphi^m_xu\|.
    \end{displaymath}
    The dual-cocycle term has zero exponential growth in both time directions, by the definition of the zero Oseledets subspace, i.e.
    \begin{displaymath}
        \lim_{m\to\pm\infty}\frac1m
        \log\|\widehat C_x^{(m)}b_x(u)\|=0.
    \end{displaymath}
    Taking logarithms in the comparison above and dividing by $|m|$ yields
    \begin{displaymath}
        \frac{\log c}{|m|}+\frac{\log\|D\varphi^m_xu\|}{|m|}
        \leq\frac{\log\|\widehat C_x^{(m)}b_x(u)\|}{|m|}
        \leq\frac{\log C}{|m|}+\frac{\log\|D\varphi^m_xu\|}{|m|}.
    \end{displaymath}
    Since $\frac{\log c}{|m|}\to0$ and $\frac{\log C}{|m|}\to0$ as $|m|\to\infty$, the middle term has zero exponential growth in both time directions, and so does $\|D\varphi^m_xu\|$.  This proves \eqref{eq:tangent-zero-exponents}.
\end{proof}

This conclusion concerns $TN$, not the entire ambient tangent bundle $TP|_N$. Nonzero transverse exponents can still occur. In particular, the argument does not rule out the existence of an invariant measure on a hyperbolic periodic orbit.

\subsection{Positive entropy and the tangent directions of a carrier}\label{sec:entropy-contradiction}

The last step is an entropy statement on the compact manifold $\Sigma$. The following lemma relates positive entropy on $\Sigma$ to the tangent directions of the possibly noncompact carrier $N$.

\begin{lemma}\label{lem:entropy-tangent}
    Let $g: Q\to Q$ be a $C^2$ diffeomorphism of a compact smooth manifold without boundary, and let $\nu$ be a $g$-ergodic Borel probability measure with $h_\nu(g)>0$. Let $E_g^u(x)$ be the sum of its strictly positive Lyapunov subspaces. If an embedded $C^1$ submanifold $S\subset Q$ satisfies $\nu(S)=1$, then
    \begin{equation}\label{eq:unstable-tangent-intersection}
        T_xS\cap E_g^u(x)\ne\{0\}
        \quad\text{for }\nu\text{-almost every }x.
    \end{equation}
    More generally, without the assumption $\nu(S)=1$, the set of regular points $x\in S$ with $T_xS\cap E_g^u(x)=\{0\}$ has $\nu$-measure zero. No compactness, closedness, or invariance of $S$ is required.
\end{lemma}
\begin{proof}
    We first explain the precise consequence of the entropy-dimension formula that we need. Let $\lambda_1>\cdots>\lambda_r>0$ be the distinct positive Lyapunov exponents. Pesin theory gives full unstable manifolds on a full-measure regular set. Ledrappier--Young's conditional-dimension theorem and entropy formula give nonnegative numbers $\gamma_i$ with
    \begin{displaymath}
        h_\nu(g)=\sum_{i=1}^r\lambda_i\gamma_i,
        \qquad \delta^u=\sum_{i=1}^r\gamma_i,
    \end{displaymath}
    where each $\gamma_i$ is the unstable-dimension contribution of the conditional measures in the direction of the $i$-th positive exponent, and $\delta^u$ is the almost-sure pointwise dimension of the conditional measures along full unstable leaves \cite[Section~7.3, Proposition~7.3.1, and Theorem~C$'$]{LedrappierYoung1985b}. These results allow for zero Lyapunov exponents and do not require the unstable conditional measures to be absolutely continuous with respect to leaf volume. They therefore apply to time-one maps of smooth flows, which may have zero Lyapunov exponents, in particular in the flow direction. The positive-entropy hypothesis gives $\delta^u>0$.

    Choose a measurable partition $\eta$ subordinate to full unstable manifolds, as in \cite[Section~7.1 and Lemma~9.1.1]{LedrappierYoung1985b}, and disintegrate
    \begin{displaymath}
        \nu=\int\nu_\zeta\,d\overline\nu(\zeta).
    \end{displaymath}
    For almost every partition element $\zeta$, its conditional probability $\nu_\zeta$ is concentrated on one unstable leaf, and for $\nu_\zeta$-almost every $y$, its leafwise local dimension is $\delta^u$. In particular, $\nu_\zeta$ is not atomic. To verify this implication, suppose $\nu_\zeta(\{y\})=a>0$.  For intrinsic unstable balls, $\nu_\zeta(B^u(y,r))\to a$ as $r\downarrow0$, so
    \begin{displaymath}
        \lim_{r\downarrow0}
        \frac{\log\nu_\zeta(B^u(y,r))}{\log r}=0,
    \end{displaymath}
    whereas the almost-sure value of this limit on the regular set is $\delta^u>0$. An atom of positive mass cannot belong to a conditional null exceptional set. This contradicts $\delta^u>0$.  No absolute continuity of the unstable conditional measures is being assumed.

    Let $\mathcal R$ be a common full-measure set on which the following hold simultaneously: the Oseledets splitting exists and the Lyapunov exponents, together with their multiplicities, are constant; the Pesin unstable manifolds are defined and depend measurably on the point; and the Ledrappier--Young conditional-dimension conclusions and that non-atomicity of the unstable conditional measures are valid. Put
    \begin{displaymath}
        B=\{x\in S\cap \mathcal R:T_xS\cap E_g^u(x)=\{0\}\}.
    \end{displaymath}

    This set is measurable: $S$ is Borel, $x\mapsto T_xS$ is continuous on $S$, and $x\mapsto E_g^u(x)$ is measurable on $\mathcal R$; trivial intersection is a measurable condition on the relevant Grassmannians, since it is equivalent to the identity
    \begin{displaymath}
        \dim(T_xS)+\dim E_g^u(x)
        =\dim\left(T_xS+E_g^u(x)\right).
    \end{displaymath}

    Take $x\in B$.  In an ambient neighborhood $V$ of $x$, express $S$ as
    \begin{displaymath}
        S\cap V=\rho^{-1}(0),\qquad
        \rho:V\longrightarrow\mathbb R^c,
        \qquad \ker D\rho_x=T_xS,
    \end{displaymath}
    where $\rho$ is a $C^1$ submersion.  On a small embedded unstable plaque through $x$, the derivative of the restricted map $\rho$ is injective, since its kernel is $T_xS\cap E_g^u(x)=\{0\}$. Write $k=\dim E_g^u(x)$.  Taking $k$ suitable linear combinations of the components of $\rho$, we obtain a linear map $L:\mathbb R^c\to\mathbb R^k$ such that the derivative of $L\circ\rho$ restricted to the plaque is invertible at $x$. By the inverse function theorem in leaf coordinates, $x$ is the only zero of this restricted map in a sufficiently small intrinsic leaf neighborhood. Every point of $S$ is a zero of $L\circ\rho$. Consequently, $x$ is an isolated intersection of $S$ with its unstable leaf.

    Now fix a full unstable leaf $W^u$. At each regular point $x\in B\cap W^u$, this leaf agrees with the unstable leaf through $x$ and has tangent space $E_g^u(x)$. The preceding argument shows that every point of $B\cap W^u$ is isolated in the intrinsic leaf topology. A Pesin leaf is a second-countable immersed manifold, for instance, by its construction as a countable union of iterated local disks. Every discrete subset of a second-countable space is at most countable: choose for each point the least member of a fixed countable basis that contains it and no other point of the subset. Distinct points receive distinct basis members. Thus $B\cap W^u$ is at most countable. We do not need the whole intersection $S\cap W^u$ to be discrete or the leaf to be properly embedded.

    For almost every $\eta$-element $\zeta$, the set $B\cap\zeta$ is countable, since $\zeta$ lies in a single unstable leaf, and $\nu_\zeta$ is not atomic. Hence $\nu_\zeta(B)=0$.  Disintegration gives
    \begin{displaymath}
        \nu(B)=\int_{Q/\eta}\nu_\zeta(B)\,d\overline{\nu}(\zeta)=0,
    \end{displaymath}
    where $\overline{\nu}$ is the pushforward probability on the quotient space $Q/\eta$.

    This proves the general assertion. When $\nu(S)=1$, its complement in the regular part of $S$ has full $\nu$-measure, so the set of regular points of $S$ with $T_xS\cap E_g^u(x)=\{0\}$ is $\nu$-null, which yields \eqref{eq:unstable-tangent-intersection}.
\end{proof}

\subsubsection{Completion of the Hamiltonian proof}

We can now finish the proof of Theorem~\ref{thm:hamiltonian-entropy}. Recall that $\mu$ was an arbitrary $f$-ergodic probability and $N$ its minimal analytic carrier. Suppose, for a contradiction, that $h_\mu(f)>0$. The map $f$ is an analytic, in particular $C^2$, diffeomorphism of the compact boundaryless manifold $\Sigma$. The submanifold $N$ is embedded and $\mu(N)=1$.  We may therefore apply Lemma~\ref{lem:entropy-tangent} with
\begin{displaymath}
    Q=\Sigma,\qquad g=f,\qquad \nu=\mu,\qquad S=N.
\end{displaymath}
Indeed, the hypotheses of Lemma~\ref{lem:entropy-tangent} are satisfied: $\Sigma$ is a compact smooth manifold without boundary; $f$ is a $C^2$ diffeomorphism of $\Sigma$; $\mu$ is an $f$-ergodic Borel probability measure with $h_\mu(f)>0$; and $N\subset\Sigma$ is an embedded $C^1$ submanifold with $\mu(N)=1$.

For almost every $x$, there exists a nonzero vector
\begin{displaymath}
    u\in T_xN\cap E_f^u(x).
\end{displaymath}
Choose such a point at which Proposition~\ref{prop:tangent-zero-exponents} also holds. As a nonzero vector in the sum of positive-exponent subspaces, $u$ has a strictly positive forward Lyapunov exponent for $D f$:
\begin{displaymath}
    \lim_{m\to+\infty}\frac1m\log\|Df^mu\|=\lambda>0.
\end{displaymath}
As a vector tangent to $N$, it has exponent zero by \eqref{eq:tangent-zero-exponents}:
\begin{displaymath}
    \lim_{m\to+\infty}\frac1m\log\|Df^mu\|=0.
\end{displaymath}
These two conclusions are incompatible. They concern the same derivative and the same restricted ambient norm, because $T_xN\subset T_x\Sigma$ and $Df=D\varphi^1|_{T\Sigma}$; the norm used in Lemma~\ref{lem:entropy-tangent} and in Proposition~\ref{prop:tangent-zero-exponents} is the same restriction of the ambient norm to $T_x\Sigma$.  Hence, the assumption $h_\mu(f)>0$ is impossible. Thus
\begin{displaymath}
    h_\mu(f)=0
    \quad\text{for every }f\text{-ergodic probability }\mu.
\end{displaymath}
Since $\mu$ was an arbitrary $f$-ergodic probability, this conclusion holds for all $f$-ergodic probabilities on $\Sigma$.

Finally, entropy is affine under ergodic decomposition, and the variational principle holds for the continuous map of the compact metric space $\Sigma$ \cite{Goodman1971, Walters1982}. Writing $\mathcal M_f(\Sigma)$ and $\mathcal M_f^{\mathrm{erg}}(\Sigma)$ for the invariant and ergodic invariant probabilities, respectively, we obtain
\begin{displaymath}
    \begin{split}
        h_{\mathrm{top}}(f)
 &=\sup_{\nu\in\mathcal M_f(\Sigma)}h_\nu(f)\\
 &=\sup_{\mu\in\mathcal M_f^{\mathrm{erg}}(\Sigma)}h_\mu(f)=0.
    \end{split}
\end{displaymath}
Indeed, the first equality is the variational principle applied to the continuous map $f$ of the compact metric space $\Sigma$.  For the second equality, every invariant probability $\nu$ admits an ergodic decomposition $\nu=\int\nu_x\,d\nu(x)$ with $f$-ergodic components $\nu_x$, and the entropy is affine under this decomposition,
\begin{displaymath}
    h_\nu(f)=\int h_{\nu_x}(f)\,d\nu(x).
\end{displaymath}
Since every $f$-ergodic probability has zero entropy, every invariant probability has zero entropy, and the supremum over invariant probabilities equals the supremum over ergodic invariant probabilities, which is zero. This proves Theorem~\ref{thm:hamiltonian-entropy}.

\subsection{Application to geodesic flows}\label{sec:geodesic-application}

We prove Theorem~\ref{thm:geodesic-entropy}. The cotangent bundle carries its canonical real-analytic symplectic form, and the real-analytic metric makes
\begin{displaymath}
    H(x,p)=\tfrac12g_x^{-1}(p,p)
\end{displaymath}
a real-analytic function. The set $\Sigma=S^*M=H^{-1}(1/2)$ is compact by local triviality, since it is a fiber bundle over the compact base $M$ with compact fiber $S^{n-1}$. It is connected when $n\geq2$, because the base and the fiber $S^{n-1}$ are connected. It is a regular energy hypersurface: the vertical derivative of $H$ at $p\ne0$ is the nonzero linear form $\dot p\mapsto g_x^{-1}(p,\dot p)$.

The assumed integrals on $U$ satisfy precisely the hypotheses of Theorem~\ref{thm:hamiltonian-entropy}. Therefore
\begin{displaymath}
    h_{\mathrm{top}}(\varphi^1|_{S^*M})=0.
\end{displaymath}
The Legendre transform
\begin{displaymath}
    \mathcal L:T^1M\longrightarrow S^*M,
    \qquad \mathcal L(x,v)=(x,g_x(v,\cdot))
\end{displaymath}
is a real-analytic diffeomorphism that conjugates the unit speed geodesic flow to the Hamiltonian flow on $S^*M$. Hence
\begin{displaymath}
    h_{\mathrm{top}}(\phi^1)=h_{\mathrm{top}}(\varphi^1|_{S^*M})=0,
\end{displaymath}
which proves the theorem.

\begin{remark}\label{rem:analytic-versus-flat}
    The proof does not show that the singular set of the original moment map has zero measure for every invariant probability. A minimal carrier $N$ may be entirely contained in that singular set. What holds is that a proper, relatively closed, flow-invariant analytic subset of this particular $N$ has measure zero.

    The crucial normal expansion consequence of analyticity is that a nonzero integral cannot vanish to infinite normal order. If $A|_N=0$ and $A\not\equiv0$, then some finite normal term $P_A$ survives. For the complete family obtained in Proposition~\ref{prop:finite-elimination}, whose fiber differentials are independent on the sets provided by Lemma~\ref{lem:jacobian-criterion}, the Poisson identities, equivariance, and the zero total weight property together force the tangent Lyapunov exponents of the carrier to vanish. For a smooth function, the finite-order starting point can fail:
    \begin{displaymath}
        \rho(z)=
        \begin{cases}
            e^{-1/z^2},&z\ne0,\\
            0,&z=0
        \end{cases}
    \end{displaymath}
    is nonzero and smooth, but every derivative at zero vanishes. There is then no first nonzero normal polynomial to which the argument could be applied. This explains the obstruction to carrying the proof over to smooth Liouville integrability without an additional hypothesis.
\end{remark}

\section*{Acknowledgments}
This research is partially supported by the National Key R\&D Program of China grant number 2022YFA1005900.

Fei Liu is partially supported by the Natural Science Foundation of China under Grant No.~12571179.

Xiaokai Liu would love to thank Professor Raul Ures and Professor Jana Hertz for their invaluable guidance and continuous support at Southern University of Science and Technology.

Xiang Zhang is partially supported by the National Natural Science Foundation of China under Grant No.~12471169.

\end{document}